\documentclass[11pt,reqno]{amsart}
\usepackage{amsaddr}
\usepackage[T1]{fontenc}
\usepackage[notrig]{physics}
\usepackage{tikz}
\usepackage{pgfplots}
\pgfplotsset{compat=1.18}
\usetikzlibrary{intersections, backgrounds}
\usepackage{graphicx,enumerate,amssymb,bbold}
\usepackage[ruled, linesnumbered, noend,vlined]{algorithm2e}
\usepackage[a4paper, left=1.3in, right=1.3in, top=1.3in, bottom=1.3in]{geometry}
\usepackage{natbib}
\usepackage{inputenc}
\usepackage{bbm}
\usepackage{subcaption}
\usepackage{placeins}
\usepackage{enumitem}
\usepackage{booktabs}
\usepackage{array}
\usepackage{microtype} 

\usepackage[colorlinks=true, pdfstartview=FitV, linkcolor=blue, 
            citecolor=blue, urlcolor=blue]{hyperref}
\usepackage{cleveref}

\numberwithin{equation}{section}
\numberwithin{figure}{section}
\theoremstyle{plain}
\newtheorem{theorem}{Theorem}[section]
\crefname{theorem}{Theorem}{Theorems}
\newtheorem{lemma}[theorem]{Lemma}
\crefname{lemma}{Lemma}{Lemmata}
\newtheorem{proposition}[theorem]{Proposition}
\crefname{proposition}{Proposition}{Propositions}

\crefname{Corollary}{Corollary}{Corollaries}

\theoremstyle{definition}

\crefname{definition}{Definition}{Definitions}
\newtheorem{condition}[theorem]{Condition}
\crefname{condition}{Condition}{Conditions}
\newtheorem{remark}[theorem]{Remark}
\crefname{remark}{Remark}{Remarks}

\crefname{example}{Example}{Examples}

\crefname{algocf}{Algorithm}{Algorithms}
\crefname{algocfline}{line}{lines}

\crefname{section}{Section}{Sections}
\crefname{subsection}{Subsection}{Subsections}
\crefname{appendix}{Appendix}{Appendices}
\crefname{figure}{Figure}{Figures}
\crefname{table}{Table}{Tables}
\crefname{algorithm}{Algorithm}{Algorithms}

\crefname{assumption}{Assumption}{Assumptions}

\newtheorem{assumptionalt}{Assumption}[theorem]

\newenvironment{assumptionp}[1]{
  \renewcommand\theassumptionalt{#1}
  \assumptionalt
}{\endassumptionalt}

\crefname{assumptionalt}{Assumption}{Assumptions}

\newcommand{\bitem}{\begin{itemize}}
\newcommand{\eitem}{\end{itemize}}
\newcommand{\mc}[1]{\mathcal{#1}}

\newcommand{\N}{\mathbb{N}}
\newcommand{\R}{\mathbb{R}}

\newcommand{\bpm}{\begin{pmatrix}}
\newcommand{\epm}{\end{pmatrix}}
\newcommand{\bvm}{\begin{vmatrix}}
\newcommand{\evm}{\end{vmatrix}}
\newcommand{\bsm}{\left(\begin{smallmatrix}}
\newcommand{\esm}{\end{smallmatrix}\right)}

\newcommand{\la}{\langle}
\newcommand{\ra}{\rangle}

\newcommand{\vphi}{\varphi}

\newcommand{\dvphi}{\nabla \vphi}
\newcommand{\df}{\nabla f}

\newcommand{\intdom}{\intr \dom}

\newcommand{\eins}{\mathbb{1}}

\DeclareMathSymbol{\mydiv}{\mathbin}{symbols}{"04}

\DeclareMathOperator{\Diag}{Diag}

\DeclareMathOperator{\dom}{dom}
\DeclareMathOperator{\intr}{int}

\DeclareMathOperator{\rg}{rg}

\DeclareMathOperator*{\argmin}{arg min}

\DeclareMathOperator{\arcsinh}{arcsinh}

\DeclareMathOperator{\KL}{KL}

\DeclareMathOperator{\Poi}{Poi}
\DeclareMathOperator{\Var}{Var}

\makeatletter
\renewcommand\paragraph{\@startsection{paragraph}{4}{0pt}%
  {3.25ex \@plus 1ex \@minus .2ex}
  {-1em}
  {\normalfont\itshape}*}
\makeatother

\makeatletter
\renewcommand\thefootnote{\@fnsymbol\c@footnote}
\makeatother

\title[Multilevel Bregman Proximal Gradient Descent]{ Multilevel Bregman Proximal Gradient Descent}
\author{Yara Elshiaty\textsuperscript{*,\ddag} \and Stefania Petra\textsuperscript{\ddag}}

\begin{document}

\begin{abstract}
We present the Multilevel Bregman Proximal Gradient Descent (ML BPGD) method, a novel multilevel optimization framework tailored to constrained convex problems with relative Lipschitz smoothness. Our approach extends the classical multilevel optimization framework (MGOPT) to handle Bregman-based geometries and constrained domains. We provide a rigorous analysis of ML BPGD for multiple coarse levels and establish a global linear convergence rate. We demonstrate the effectiveness of ML BPGD in the context of image reconstruction, providing theoretical guarantees for the well-posedness of the multilevel framework and validating its performance through numerical experiments.
\end{abstract}

\maketitle
\footnotetext[1]{Institute for Mathematics, Heidelberg University \; (\url{elshiaty@math.uni-heidelberg.de})}
\footnotetext[3]{Institute for Mathematics \& Centre for Advanced Analytics and Predictive Sciences (CAAPS), University of Augsburg \; (\url{stefania.petra@uni-a.de})}

\markboth{\MakeUppercase{Yara Elshiaty and Stefania Petra}}{\MakeUppercase{Multilevel Bregman Proximal Gradient Descent}}

\tableofcontents

\section{Introduction}
\addtocontents{toc}{\protect\setcounter{tocdepth}{1}}
WWe consider the constrained convex optimization problem  
\begin{equation} \label{eq:main_problem}
    f^{\min} = 
    \min_{x \in C} f(x),
\end{equation}
where $ C \subseteq \R^n$ is a closed, convex set, and $f: C \to \R$ is a differentiable convex function. 
While many first-order methods assume Euclidean smoothness of $f$, i.e., a Lipschitz-continuous gradient on $C$
\begin{equation} \label{eq:Lipschitz-condition}
    \norm{\df(x) - \df(y)} \leq L \norm{x-y}  \quad \text{for all } x,y \in C,
\end{equation}  
this assumption often fails in various problems including imaging applications.
These problems frequently exhibit structured smoothness more accurately captured in a Bregman geometry. Accordingly, we work under the \emph{relative smoothness condition} \cite{Bauschke:2017aa}
\begin{equation}\label{eq:rel-smoothness}
    D_f(x,y) \leq L D_\vphi(x,y) \quad \text{for all } x \in C ,y \in \intr C,
\end{equation}
where $D_f$ and $D_\vphi$ denote Bregman divergences associated with $f$ and a convex reference function $\vphi$ defined on $C$, respectively.

To address the problem setting of \eqref{eq:main_problem} we adopt \emph{Bregman proximal gradient descent (BPGD)}, which is a first-order iterative method, that extends the classical forward-backward splitting approach using Bregman divergences, \cite{Beck:2003aa}. Given an interior point $x \in \intr C$, and a step size $\tau > 0$, the BPGD update is given by
\begin{equation}\label{eq:BPGD-subproblem}
    x^{+}_\tau := \argmin_{u \in C} \tau \la \df(x), u - x \ra + D_\vphi(u,x).
\end{equation}

Despite its success in structured optimization, BPGD suffers two main limitations: (i) it cannot be accelerated beyond $\mathcal{O}(1/k)$ convergence under mere relative smoothness \cite{Dragomir2021Optimal}, and (ii) its standard formulation does not exploit multilevel or hierarchical structure inherent in many inverse problems.

In fact, inverse problems in imaging frequently admit \emph{natural multilevel discretizations} in the form of coarse-to-fine representations which provide a hierarchy of reduced problems.
Such structure can be exploited to reduce computational costs and improve conditioning, as demonstrated by Nash's multilevel optimization framework (MGOPT) \cite{Nash:2000}. However, adapting BPGD to this setting is non-trivial, particularly due to the challenge of handling \emph{explicit constraints} across discretization levels.

Building on this idea, multilevel approaches exploit the hierarchy of discretizations by solving coarse models with significantly fewer variables to obtain descent directions on the finer levels. This strategy offers several key advantages. First, it enhances computational efficiency by reducing the dimensionality of subproblems, especially during the early iterations where full-dimensionality accuracy is less critical. Second, optimization problems on coarser levels often exhibit improved conditioning. Together, these properties typically yield substantial acceleration in the initial phases of the optimization.

This behavior is also observed in acceleration schemes for BPGD, which often show rapid progress initially. For example, Nesterov’s acceleration to the relatively smooth setting was extended in \cite{Hanzely:2021vc}, achieving rates up to to $\mc{O}(k^{-\gamma})$ for $\gamma \in (0, 2]$, but only under additional structure in the reference function $\vphi$, beyond mere relative smoothness. In typical inverse problems with divergences like Kullback–Leibler, such structure is not available: theoretical guarantees remain limited to $\gamma = 1$, and numerical results indicate that accelerated rates are only transient, with asymptotic behavior reverting to the standard $O(1/k)$ rate \cite{Raus2024}.

\subsection*{Related work}
Nash's MGOPT algorithm \cite{Nash:2000} presented a general framework for adapting unconstrained smooth optimization methods in a multilevel setting. Constraints were subsequently incorporated in two distinct ways: (i) by allowing composite objectives of the form $f(x) + g(x)$, where $f$ is a smooth data term and $g$ encodes the constraint as an indicator function $g(x) = \begin{cases}
0, &x \in C \\ \infty, &x \neq C
\end{cases}$, or (ii) by building the constraints directly into the multilevel design. 
Existing works following the first approach \cite{ang2023mgprox,Parpas:2017,Parpas:2016_Magma,Lauga2024} assume Lipschitz continuity of the gradient $\df(x)$, an assumption that may not hold in many imaging problems.
A key example arises in Poisson linear inverse problems, which show up naturally whenever the imaging process involves counting photons arriving in the image domain \cite{Vardi_EM, Geman1984}, such as image deconvolution in microscopy and astronomy, or tomographic reconstruction in PET. 
Approaches of the second type typically develop models tailored to the specific structure of the constraints. Most notably box constraints $\{x \in \R^n: l \leq x \leq u\}$ have been analyzed in the context of trust-region and line search methods \cite{Gratton:2008,WenGoldfarb:2009}. The authors of \cite{Mueller2022} proposed a geometric multilevel approach by mapping the box constraint into a Riemannian manifold via the Hessian metric $\nabla^2 \vphi$, interpreting mirror descent (MD), a special case of BPGD, as Riemannian gradient descent, \cite{Raskutti2015}. 
While this method provides a coherent framework for handling convex constraints in a multilevel setting, verifying descent directions using coarse information is more involved in the Riemannian context and, to the best of our knowledge, lacks convergence guarantees at present. 

Our proposed algorithm, ML-BPGD, addresses the limitations of both approaches by introducing a general strategy for incorporating convex constraints into the multilevel structure with convergence guarantee, specifically tailored to data terms satisfying relative smoothness.
\subsection*{Contribution and organization}
The remainder of this paper is organized as follows. \Cref{sec:preliminaries} reviews the BPGD method and introduces the assumptions and properties that will be central to the development of ML-BPGD. In \Cref{sec:MLBPGD-big-section}, we begin with an overview of the multilevel algorithm in the unconstrained setting, which we then extend to accommodate BPGD with convex constraints. For clarity of presentation, we first focus on the case of a single coarse level to establish notation and intuition, before generalizing to the full multilevel setting. This section also introduces our algorithm, proves that it is well-defined and establishes a convergence result for the function values. Finally, \Cref{sec:NumExperiments} presents extensive numerical experiments comparing BPGD and ML-BPGD on imaging problems with inherent Bregman geometry.

\subsection*{Notation}
We denote by $\mathbb{R}^n_+ = \{ x \in \mathbb{R}^n : x_i \geq 0, \, i=1,\ldots,n \}$ the nonnegative orthant of $\mathbb{R}^n$, and by $\mathbb{R}^n_{++} = \{ x \in \mathbb{R}^n : x_i > 0, \, i=1,\ldots,n \}$ the positive orthant. For a vector $x \in \mathbb{R}^n$, we write $\mathrm{Diag}(x)$ for the $n \times n$ diagonal matrix whose $i$-th diagonal entry is $x_i$. The functions $\exp(\cdot)$ and $\log(\cdot)$ with vector argument denote the elementwise exponential and natural logarithm, respectively. We use $[n]$ to denote the set $\{1, 2, \ldots, n\}$ for $n \in \N$, and set $[n]_0 = \{0, 1, \ldots, n\}$.

\addtocontents{toc}{\protect\setcounter{tocdepth}{2}}
\section{Bregman proximal gradient descent}\label{sec:preliminaries}
We briefly recall key concepts related to Bregman divergences and relative smoothness that underpin BPGD. 

Let $\vphi: \intr \dom \vphi \to \mathbb{R}$ be a differentiable convex function. The \emph{Bregman divergence} associated with $\vphi$ is defined as
\begin{equation}
    D_\vphi(x,y) = \vphi(x) - \vphi(y) - \la \dvphi(y), x-y \ra.
\end{equation}
Whenever $\vphi$ is convex, the Bregman divergence is convex in its first argument, 
and it is strictly positive for all $x \neq y$ if $\vphi$ is strictly convex.
In this work, we assume $\vphi$ is strictly convex on the feasible set $C$. Bregman divergences are also linear with respect to the generating function; for convex $f$ and $g$, and any $\gamma \in \mathbb{R}$,
\begin{equation} \label{eq: linearity-Bregman-div}
    D_{f+\gamma g} = D_f + \gamma D_g \quad \text{on} \intr \dom f \cap \intr \dom g .
\end{equation}

A function $f$ is said to be $L$-smooth relative to $\vphi$ on $C$ if
\begin{equation}
    D_f(x,y) \leq L D_\vphi(x,y) \quad \text{for all } x \in C, y \in \intr C.
\end{equation}
This condition admits the following equivalent formulations:
\begin{enumerate}
    \item $L\vphi-f$ is convex on $C$,
    \item $f(x) \leq f(y) + \la \df(y), x-y \ra + LD_\vphi(x,y)$ for all $x, y \in \intr C$,
    \item under twice differentiability, $\nabla^2 f(x) \preceq L \nabla^2 \vphi(x)$, where $\preceq$ denotes the L\"owner partial order 
    on symmetric matrices,
\end{enumerate}
see \cite{Bauschke:2017aa,Lu2018} for detailed proofs.

Utilizing BPGD updates \eqref{eq:BPGD-subproblem} to solve \eqref{eq:main_problem} requires the capability of solving instances of a subproblem of the general form
\begin{equation} \label{eq: simplified class of BPGD problems}
    x^{+}_\tau = \argmin_{u \in C} \{\la c, u \ra + \vphi(u) \}
\end{equation}
with $c = \tau \df(x) - \dvphi(x)$. In the typical formulation and implementation of a first-order method to solve \eqref{eq:main_problem}, one selects the strictly convex reference function $\vphi$ based on the structure of the constraints defining $C$ of constraints, while also ensuring that the subproblem \eqref{eq:BPGD-subproblem} (or its more general form \eqref{eq: simplified class of BPGD problems}) can be solved efficiently.
In summary, we choose $\vphi$ such that it satisfies the following assumption.
\begin{assumptionp}{A}[Solvability of the problem]\label{ass:solveability}
    \begin{enumerate}
    \item $f$ is $L$-smooth relative to $\vphi$ on $C$,
    \item the subproblem \eqref{eq:BPGD-subproblem} always has a solution on the (relative) interior of $C$ that is a singleton and efficiently computable.
\end{enumerate}
\end{assumptionp}
\begin{remark} \label{rm:MD-special-setting}
    One special case where \cref{ass:solveability} is readily satisfied is when $\vphi$ is a Legendre function, i.e. both essentially smooth and essentially strictly convex. In this context, let $\dom \vphi = C$. In \cite[Thm 3.12]{Bauschke:1997aa}, the authors show that the Bregman projection $\bar x = \argmin D_\vphi(\cdot,y)$ of an interior point $y \in \intr C$ exists on $C$, lies in the interior, and is unique.
    Here, essential smoothness ($\norm{\dvphi(x_n)} \to \infty$ for $x_n \to x$, $x_n \in \intr \dom \vphi$, see \eqref{eq:convex-conjugate}) ensures existence, while essential strict convexity ($\dvphi^\ast(\dvphi(y)) = \{y\}$ for $y \in \dom \dvphi$) guarantees uniqueness. These results are extended in \cite{Petra2013a} to the case with an added linear term $\la l, \cdot \ra + D_\vphi(\cdot,y)$ for $\norm{l} < \infty$ which covers our setting.

    Furthermore, the Legendre property of $\vphi$ enables interpreting \eqref{eq:BPGD-subproblem} as an MD update. This is possible since the gradient mapping
    \begin{equation} \label{eq:convex-conjugate}
        \dvphi: \intr \dom \vphi \to \intr \dom \vphi^*, \quad x \mapsto \dvphi(x)
    \end{equation}
    is a topological isomorphism with inverse $(\dvphi)^{-1} = \dvphi^*$, see \cite[Theorem 26.5]{Rockafellar:1997}, where $\vphi^\ast$ denotes the convex conjugate function
    \begin{equation}
        \vphi^\ast(x^\ast) = \sup_{x \in \intr \dom \vphi} \{\la x^\ast, x \ra - \vphi(x) \}.
    \end{equation}
    Using first-order optimality criteria, and by simple reformulations, this structure allows us to rewrite \eqref{eq:BPGD-subproblem} in the MD form
    \begin{equation} \label{eq:MD}
         x^{+}_\tau = \dvphi^*(\dvphi(x)-\tau \df(x)),
    \end{equation}
    as introduced in \cite{Nemirovski1983} and explored in the context of BPGD in \cite{Beck:2003aa}. 
    The MD perspective localizes the computational burden: if $\dvphi^\ast$ has a closed form or is inexpensive to compute, then \cref{ass:solveability}.(2) is automatically satisfied, see \cref{apdx:A} for examples. 
\end{remark}
\subsection{Properties of BPGD} This section recalls a few theoretical results on the BPGD update which will serve as the foundation for the multilevel extension in \Cref{sec:MLBPGD-big-section}.

\begin{lemma}[Fixed point property of BPGD] \label{lm: fixed point property of BPGD}
    Let $x$ be a minimizer of $f$ over $C$. Then, for any $\tau > 0$, it holds that
    \begin{equation}
        x^+_\tau=x.
    \end{equation}
\end{lemma}
\begin{proof}
By definition of $x^+_\tau$, we have
\begin{equation}
    \tau \la \df(x), x^+_\tau - u \ra + D_\vphi(x^+_\tau, x) \leq  D_\vphi(u,x) \; \forall u \in C.
\end{equation}
Choosing $u = x$, this reduces to
\begin{equation}
    D_\vphi(x^+_\tau, x) \leq \tau \la \df(x), x - x^+_\tau \ra .
\end{equation}
The statement follows from the non-negativity of the Bregman divergence and the first-order optimality condition $\la \df(x), x - u \ra \leq 0$ for all $u \in C$, which holds since $x$ is a minimizer of $f$ over $C$.
\end{proof}

Henceforth, we assume that \cref{ass:solveability} holds. We now state a key result for BPGD.

\begin{lemma}[Sufficient descent of BPGD, \cite{Teboulle_2018}] \label{lm:sufficient-descent-BPGD} Let $x \in \intr C$. For any $\tau \in (0, L^{-1}]$, the following inequality holds
    \begin{equation}
        f(x^+_\tau) \leq f(x) - \frac{D_\vphi(x, x^+_\tau)}{\tau}.
    \end{equation}
\end{lemma}
The BPGD scheme generates the sequence $\{ x^k \}_{k \in \N}$, where $x^k := \left(x^{k-1}\right)^+$, starting from an initial point $x^0 \in \intr C$. 
Under the aforementioned setting, the function values of BPGD converge sublinearly. Specifically, one can prove that after $k$ iterations, using a step size $\tau = L^{-1}$ the following holds for any $x \in C$:
\begin{equation}
    f(x^k) - f(x) \leq \frac{LD_\vphi(x, x^0)}{k},
\end{equation}
see \cite{Nemirovski1983, Teboulle_2018}. Analogous to the Euclidean case of primal gradient descent, a linear convergence rate can be attained under a Polyak-\L{}ojasiewicz (PL)-type inequality adapted to the Bregman setting. This inequality bounds the Bregman distance of a BPGD iterate and its initial point to the minimum of $f$. The following assumptions were first introduced in \cite{Bauschke:19} in the context of MD. 

\begin{assumptionp}{B}[Polyak-\L{}ojasiewicz-like condition]\label{ass:PL-with-L-dependence}
    There exists a function $\theta: \R_{++} \to \R_{++}$ and a scalar $\eta > 0$ such that
    \begin{equation} \label{eq:dependence on tau}
        D_\vphi(x,x_\tau^+) \geq \theta(\tau) D_\vphi(x, x_1^+)
    \end{equation}
    and 
    \begin{equation} \label{eq: Pl-cond}
        D_\vphi(x, x^+_1) \geq \eta (f(x) - f^{\min})
    \end{equation}
    for all $x \in \intr C$.
\end{assumptionp}

\begin{remark}      
\begin{enumerate}[label=(\alph*), leftmargin=*, align=left, itemsep=3pt]
    \item (\textbf{Scaling condition \eqref{eq:dependence on tau}})  
    The function $\theta$ quantifies the scaling of the Bregman divergence in dependence on the step size $\tau$. In the case $\vphi = \frac{1}{2}\norm{\cdot}^2_2$, $\theta$ is a quadratic function due to the homogeneity of norms. This, however, does not hold for general divergences.  
    A sharper definition of $\theta$ would also depend on $x$ rather than assuming uniformity. In practice, verifying \eqref{eq:dependence on tau} is often challenging and typically requires an explicit expression for the update $x_\tau^+$, something not even guaranteed in the MD case, see \cite{Bauschke:19} and the examples therein.

    \item (\textbf{PL-inequality \eqref{eq: Pl-cond}})  
    The condition \eqref{eq: Pl-cond} simplifies to the known PL-inequality  
    \begin{equation}
        \frac{1}{2} \norm{\df(x)}_2^2 \geq \eta (f(x) - f^{\min})
    \end{equation}
    in the special case $\vphi = \frac{1}{2}\norm{\cdot}^2_2$.  
    Moreover, if $f$ is $\mu$-strongly convex relative to $\vphi$, i.e.,  
    \begin{equation}
        f - \mu \vphi \quad \text{ is convex on } C,
    \end{equation}
    then for $\mu > 1$, the result \cite[Lemma 3.3]{Bauschke:19} establishes \eqref{eq: Pl-cond} in the setting where $\vphi$ is Legendre. The proof relies on the identity of the gradient envelope
    \begin{equation}
        \min_u \tau \la \df(x), u-x \ra + D_\vphi(u,x) = - D_\vphi(x, x_\tau^+),
    \end{equation}
    which follows from the three-point identity \cite{Chen_Teboulle_1993}. Crucially, this derivation does not depend on $\vphi$ being Legendre. Hence, \eqref{eq: Pl-cond} follows directly from the convexity of $f-\vphi$.

    \item (\textbf{On the order of Bregman arguments})  
    The ordering of the Bregman arguments in \cref{ass:PL-with-L-dependence} is not canonical. One could reverse the arguments and still maintain a valid generalization of the Euclidean case.
    For Bregman divergences that are not entirely asymmetric, i.e., when  
    \begin{equation}
        \alpha(\vphi) := \inf \left\{ \frac{D_\vphi(x,y)}{D_\vphi(y,x)} : x,y \in \intdom \vphi, x \neq y \right\} \in [0,1],
    \end{equation}
    with $\alpha(\vphi) \neq 0$, the ordering affects the expression only by a constant factor, rendering the choice largely inconsequential. However, in fully asymmetric cases such as the log-barrier function (cf. \cref{apdx:A}), where $\alpha(\vphi) = 0$, the order matters. Our adopted ordering aligns with the definition of the BPGD subproblem \eqref{eq:BPGD-subproblem} and the connection of \eqref{eq: Pl-cond} to relative strong convexity.
\end{enumerate}
\end{remark}

Examples of objective and reference function pairs
$(f,\vphi)$ satisfying the Bregman Polyak-\L{}ojasiewicz-like condition can be found in \cite{Bauschke:19}.

Using \cref{ass:PL-with-L-dependence}, one can prove the linear convergence rate as follows.
\begin{lemma}
    For the objective $f$ and the prox function $\vphi$, let $\{x^k\}_{k \in \N}$ be the sequence generated using constant step size $\tau \in (0, L^{-1}]$. Assume \cref{ass:solveability,ass:PL-with-L-dependence} hold. Then, defining $r := \frac{\theta(\tau)\eta}{\tau}$, it holds that
    \begin{equation}
        f(x^{k}) - f^{\min} \leq (1-r)^k (f(x^0) - f^{\min}),
    \end{equation}
    and $r \in (0,1]$.
\end{lemma}
\begin{proof}
    We give a short proof for completeness, adapting slightly from \cite{Bauschke:19}. By \cref{lm:sufficient-descent-BPGD} and \cref{ass:PL-with-L-dependence}, we have:
    \begin{equation}
        f(x^{k+1}) \leq f(x^k) - \frac{D_\vphi(x^k, x^{k+1})}{\tau} \leq f(x^k) - \frac{\theta(\tau)\eta}{\tau} (f(x^k)-f^{\min}).
    \end{equation}
    A short reformulation yields
    \begin{equation} \label{eq:fine-level-recursion}
        f(x^{k+1}) - f^{\min} \leq (1-r) (f(x^k)-f^{\min}).
    \end{equation}
    Since $r > 0$ and the sequence $\{f(x^k)-f^{\min}\}_{k \in \N}$ is non-increasing, $r \leq 1$ must hold. A recursion of \eqref{eq:fine-level-recursion} yields the statement.
\end{proof}

\section{Multilevel Bregman proximal gradient descent}\label{sec:MLBPGD-big-section}
For large-scale problems, multilevel optimization reduces dimensionality when far from the solution, where full high-resolution information is not yet critical. It does so by defining a coarse problem representation with significantly fewer variables, which is used to compute a descent direction for the finer level, thereby accelerating convergence.
\subsection{Overview of unconstrained multilevel optimization}\label{sec:unconstrained-ML}
We provide an overview of Nash's MGOPT framework, \cite{Nash:2000}, focusing on the two-grid cycle for updating $x^{k+1}$ from the current iterate $x^k$. This update involves either a search direction obtained from a coarse-grid model with fewer variables (coarse correction) or, when the coarse correction is ineffective, a standard local approximation defined on the fine grid (fine correction). We denote such an update iteration by $\rho: \R^n \to \R^n$. This approach is summarized in \cref{alg:Two-level-algorithm}, where $\rho$ is a general iteration update; later it will denote a BPGD step.

Let $n$ be the dimension of the full-problem, which we henceforth call the \emph{fine} dimension. We assume access to a convex coarse version of the fine objective $f$, denoted by $f_H$, defined on $\R^{n_H}$ with $n \gg n_H$.
Furthermore, we assume linear maps $R: \R^{n} \to \R^{n_H}$ (restriction)
and $P: \R^{n_H} \to \R^n$ (prolongation) are provided to transfer points between levels, typically via interpolation. We impose the standard \emph{variational property} (or Galerkin condition) $R = cP^\top$ for some positive scalar $c$, see \cite{SIAM_Multigrid}. In this paper we set $c = 1$ for simplicity. 
\begin{remark}
    This geometric approach, which assumes the problem has a "natural" coarse instance, common in variational problems derived from infinite-dimensional formulations, contrasts with algebraic methods, where the problem is evaluated solely on the fine grid, and coarse-grid representations are constructed using the restriction and prolongation mappings, cf. \cite{ang2023mgprox,Parpas:2017}. Importantly, the theoretical framework we present and expand upon here does not contradict an algebraic approach, which can still be employed if desired.
\end{remark}

\begin{algorithm}[H]
        \caption{Two level optimization}\label{alg:Two-level-algorithm}
        \DontPrintSemicolon
        \textbf{initialization: $x^0 \in \R^n$}\;
        \Repeat{a stopping rule is met.} {
          \If{condition to use coarse model is satisfied at $x^k$} {
            $x_H^k = Rx^k$ \;
            $x_H^{\min, k} = \argmin_{x \in \R^{n_H}} \psi^k(x)$ \tcc*{solve coarse model}
            $d_k = P(x^{\min, k}_H - x_H^k)$ \tcc*{compute descent direction }
            Find $\alpha_k > 0$ such that $f(x^k + \alpha_k d_k) \leq f(x^k)$ \tcc*{line search}
            $z^{k+1} = x^k + \alpha_k d_k$  \tcc*{coarse correction}
            Apply fine-grid iteration $x^{k+1} = \rho(z^{k+1}; f)$ \tcc*{post-smoothing}
          }\Else{
          Apply fine-grid iteration $x^{k+1} = \rho(x^{k}; f)$.
          }
          Increment $k \leftarrow k+1$.
          }
    \end{algorithm}

\paragraph{Coarse model.}
The core idea is to define the coarse model by linearly modifying the coarse function $f_H$. At each iteration $k \in \N$, it is constructed based on the current iterate $x^k$ as
\begin{equation} \label{eq:k-th-coarse-function}
    \psi^k(x) := f_H(x) + \la v^k, x - Rx^k \ra, \quad \text{with} \quad  v^k := R \df(x^k) - \df_H(Rx^k),
\end{equation}
to define the coarse problem
\begin{equation} \label{eq:coarse-model-objective}
    \argmin_{x \in \R^{n_H}} \psi^k(x).
\end{equation}
For the initial coarse-grid iterate $x_H^k = Rx^k$, the gradient of the coarse model satisfies the \emph{first-order coherence}
\begin{equation} \label{eq:first-order-coherency}
    \nabla \psi^k(x_H^k) = R\df(x^k),
\end{equation}
which ensures that the restriction of a critical point remains critical for the coarse problem.

The coarse model is designed to efficiently compute a descent direction $d^k$ for $f$ at the point $x^k$, making use of the lower-dimensional and thus computationally cheaper coarse variables. To this end, one solves \eqref{eq:coarse-model-objective} to find the minimizer $x_H^{\min,k}$, or more commonly, an approximate solution $x_H^{+,k}$ satisfying $\psi^k(x_H^{+,k}) < \psi^k(Rx^k)$. This approximate solution is often obtained via an iterative update rule similar to the post-smoothing step applied to the fine objective.

\subsection{Two-level BPGD}\label{sec:2l-BPGD}
Building upon the unconstrained multilevel framework, we now extend the method to handle convex constraints within a multilevel variant of Bregman proximal gradient descent (BPGD). To ease understanding, we start with the two-level scheme, highlighting the key concepts and summarizing the approach in \cref{alg:Two-level-BPGD}. The full multilevel scheme is discussed in \Cref{sec:MLBPGD}.
\subsubsection{Coarse model}\label{sec:2l-BPGD-coarse-model}
At iteration $k$, the coarse constrained problem is constructed similarly to the unconstrained setting, while ensuring consistency between the coarse and fine feasibility sets. Specifically, we define closed and convex sets $C_H^k \subseteq \mathbb{R}^{n_H}$,
referred to as \textit{coarse constraints},
which satisfy the \textit{coarse feasibility consistency property}
\begin{equation} \label{eq:coarse-constraint-set}
    C_{H}^k \subseteq \{w \in \R^{n_H}: x^k + P(w - x_H^k) \in C \},
\end{equation}
ensuring that every coarse-feasible point prolongs to a fine-feasible one.
The \textit{coarse minimization problem} at iteration $k$ is defined as
\begin{equation} \label{eq:k-th-coarse-problem}
    \min_{x \in C_{H}^k} \psi^k(x) \quad \psi^k \text{ as defined in } \eqref{eq:k-th-coarse-function}.
\end{equation}
Note that $Rx^k$ is trivially contained in $C_H^k$. The first-order coherence \eqref{eq:first-order-coherency} and variational property, paired with \eqref{eq:coarse-constraint-set}, extends the consistency of transferring critical points from the fine to the coarse level to the constrained setting.
\begin{lemma}\label{lm:transferability-of-criticality}
        If $x^k$ is a critical point of the fine objective $\{f(x): x \in C\}$, then $Rx^k$ is a critical point of $\{\psi^k(x): x \in C^k_H\}$.
\end{lemma}
\begin{proof}
    The variational property of the transfer operators and the first-order coherence yield
    \begin{equation}
        \la \nabla \psi(Rx^k), w-Rx^k \ra = \la \df(x^k), P(w-Rx^k).
    \end{equation}
    The first-order optimality condition guarantees that $x^k$ is a critical point of $f$ over $C$ iff
    \begin{equation}
        \la \nabla f(x^k), d \ra \geq 0 
    \end{equation}
    for all feasible directions such that $x^k + d \in C$. By the definition of the coarse constraints \eqref{eq:coarse-constraint-set}, $w \in R^{n_H}$ is chosen such that $x^k + P(w-Rx^k) \in C$, completing the proof.
\end{proof}

We employ BPGD schemes both as a post-smoothing step on the fine level and to approximately solve the coarse problem 
$\psi^k$ see \eqref{eq:k-th-coarse-problem}, especially when the dimension 
$n_H$ is too large for fast convergence. Specifically, we perform the update:
\begin{equation} \label{eq:BPGD-coarse-level}
    \argmin_{u \in C_H^k} \Phi(u; x, \tau_H, \psi^k, \vphi_H) : = \argmin_{u \in C_H^k} \{ \tau_H \la \nabla \psi^k(x) - \dvphi_H(x), u \ra + \vphi_H(u) \}.
\end{equation}
Thus, the same challenges in efficiently applying BPGD updates arise on the coarse level as on the fine level, primarily in selecting an appropriate reference function $\vphi_H: \R^{n_H} \to (-\infty, \infty]$ tailored to the coarse objective $\psi^k$. Due to the linearity of the Bregman divergence, it suffices to focus on choosing $\vphi_H$ to suit $f_H$.
In summary, analogously to \cref{ass:solveability}, we choose a reference function $\vphi_H$ such that
\begin{assumptionp}{C}[solvability of coarse problem] \label{assump:solvability-coarse-2l}
    \begin{enumerate}
    \item $f_H$ is $L_H$-smooth relative to $\vphi_H$ on $C_H^k$ for all $k \in \N$,
    \item the subproblem \eqref{eq:BPGD-coarse-level} always has a solution on $C_H^k$ that is a singleton and efficiently computable.
\end{enumerate}
\end{assumptionp}


\begin{algorithm}[h]
\DontPrintSemicolon
\textbf{initialization: $x^0 \in \intr C, k = 0$}\;
\Repeat{a stopping rule is met} {
  \If{condition to use coarse model is satisfied at $x^k$} {
    $x_H^{0,k} = Rx^k$ \;
    \For{$i \in [1, \cdots, m]$} {
    $x_H^{i,k} = \argmin_{u \in C_H^k} \Phi(u; x_H^{i-1,k}, \tau_H, \psi^k, \vphi_H)$ \tcc*{\small solve coarse problem}
    }
    $d_k = P(x^{m, k}_H - x_H^{0,k})$ \tcc*{\small descent direction }
    Find $\alpha_k \in (0,1]$ such that $f(x^k + \alpha_k d_k) \leq f(x^k)$ \tcc*{line search}
    $z^{k+1} = x^k + \alpha_k d_k$  \tcc*{\small coarse correction}
    $x^{k+1} = \argmin_{u \in C} \Phi(u; z^{k+1}, \tau, f, \vphi)$ \tcc*{\small post-smoothing}
  }\Else{
  $x^{k+1} = \argmin_{u \in C} \Phi(u; x^k, \tau, f, \vphi)$ \tcc*{\small fine correction}
  }
  Increment $k \leftarrow k+1$
  }
\caption{Two-level BPGD}\label{alg:Two-level-BPGD}
\end{algorithm}

\begin{figure}
    \centering
    \makebox[\linewidth][r]{
        \includegraphics[width=0.8\linewidth]{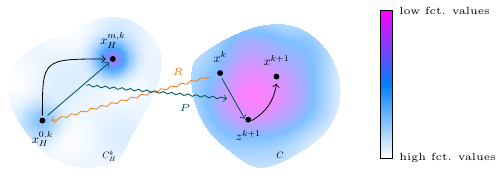}
    }
    \caption{Flowchart of \cref{alg:Two-level-BPGD}. Cooler and lighter colors refer to bigger function values.}
    \label{fig:alg-flowchart}
\end{figure}

A full multilevel scheme applies \cref{alg:Two-level-BPGD} recursively, where $n_H$ takes on the role of the fine dimension and $\psi^k$ are the fine objectives of their respective iterations. We direct the reader to \Cref{sec:MLBPGD} for a detailed description of the V-cycle variant.

\subsubsection{Coarse correction condition}\label{sec:2l-BPGD-CCC}

As shown in \cref{lm:transferability-of-criticality}, if the fine iterate $x^k$ is stationary for the fine objective $f$, then its restriction $Rx^k$ is also stationary for the coarse model $\psi^k$. To ensure that the algorithm makes full use of the problem's dimensionality near such points, we deliberately avoid relying on coarse-level information in their vicinity. However, the converse does not hold: a nonstationary fine-level iterate may yield a stationary point on the coarse level. In such cases, constructing and solving the coarse model incurs unnecessary computational cost, and should be avoided. To address this, we employ a coarse correction condition that prevents ineffective coarsening, following the approach of \cite{Gratton:2008,Conn2000}.

The constants $\kappa$ and $\epsilon$ represent the tolerance on the first-order optimality conditions.
\begin{condition}[Coarse correction criteria]
    The iterate $x^k$ triggers a coarse correction step if the following holds
    \begin{equation}\label{eq:constrained-CCC}
        \abs{ \min_{x_H^{0,k} + d \in C^k_{H}} \la \nabla \psi^k(x_H^{0,k}), d \ra} \geq \kappa \abs{\min_{x^k + d \in C} \la \df(x^k), d \ra} \geq \epsilon
    \end{equation}
    with $\kappa, \epsilon \in (0,1)$, and
    \begin{equation} \label{eq:last-iterate-CCC}
        D_\vphi(x^k ,\widetilde{x}) \geq \epsilon_x,
    \end{equation}
    where $\epsilon_x > 0$ and $\widetilde{x}$ is the last iterate that triggered a multilevel step. 
\end{condition}

The condition \eqref{eq:last-iterate-CCC} prevents a coarse correction if the current iterate is very close to $\widetilde{x}$, since a new coarse correction would lead to values similar to the last step. The nearness of the points is measured in Bregman distance in accordance with the non-Euclidean setting of our problem. 

The criticality measure \eqref{eq:constrained-CCC} extends standard criteria from the unconstrained setting, where a coarse correction is triggered if $\norm{\nabla \psi^k(x_H^k)}_2 \geq \kappa  \norm{\df{(x^k)}}_2$, cf. \cite{Gratton:2008}. However, \eqref{eq:constrained-CCC} is not numerically viable, as it requires minimizing over all feasible directions. Depending on the structure of the feasible set, this condition can be relaxed. In practice, the unconstrained criterion often serves as a reliable proxy and is computationally more efficient. Alternative coarse criticality measures involving Euclidean projections have also been proposed, though these tend to be costly when dealing with complex or nontrivial constraints.

\subsubsection{Transfer operators}\label{sec:2l-BPGD-transfer-ops}
Constructing linear operators $P$ and $R$ to transfer points between levels has been extensively studied in multigrid methods for PDEs \cite{Hackbusch:1985,SIAM_Multigrid,Trottenberg:2001}, and many of these ideas extend naturally to multilevel optimization. Classically, the prolongation operator $P$ is defined by interpolation with weights tailored to the underlying problem's structure. A simple choice is linear interpolation via convolution with the (normalized) kernel
\begin{equation} \label{eq:1D-kernel}
    K_{1\text{D}} := \frac{1}{4} \begin{bmatrix} 1 \quad 2 \quad 1 \end{bmatrix}.
\end{equation}
The corresponding restriction operator $R = P^\top$ then maps from a fine grid of size $n$ to a coarse grid of size $\frac{n}{2}$ for even $n$ and $\frac{n-1}{2}$ for odd $n$. While edge asymmetry disrupts the balance of the weights for even $n$, the transfer operators are perfectly balanced for odd dimensional input and as such sum preserving in that case. Additionally, $P$ has full rank and a trivial null space.

In this work, we focus primarily on square domains with dimensions $n = (2^m-1)^2$ and $n_H = (2^{m-1}-1)^2$ for fine and coarse levels, respectively. In such cases, bilinear interpolation is a natural choice. It can be implemented via a 2D convolution kernel, for instance,
\begin{equation} \label{eq:2D-kernel}
    K_{2\text{D}} := K_{1\text{D}} \otimes K_{1\text{D}} = \frac{1}{16} \begin{bmatrix}
        1 \quad 2 \quad 1 \\ 2 \quad 4\quad 2 \\ 1 \quad 2 \quad 1 
    \end{bmatrix}.
\end{equation}
This kernel is also sum-preserving and yields a prolongation operator $P$ with full rank. Higher-order interpolation schemes, such as bicubic or spline-based interpolators, are also common and may offer improved accuracy. However, the selection of transfer operators in the context of nonlinear optimization problems is highly problem-dependent, just as in PDE-specific multigrid settings. Our future work will also focus on the impact of refined choices of the the grid transfer mappings $P$ and $R$. The present form of $R$ may be viewed as a primitive low-pass operator from the viewpoint of image signal processing, whose role in the present context of optimization differs, however. We conjecture that suitable adaptive choices depending on the objective function are beneficial.

\subsubsection{Feasibility of the two-level BPGD}\label{sec:2l-BPGD-well-defined}
By ensuring the consistency between the coarse and fine constraints, we have $x^k + \alpha^k d^k \in C$ for any $\alpha^k \in (0,1]$, given that $x^k + d^k \in C$ and $C$ is convex. Moreover, since $x^k$ is in the (relative) interior of $C$ and $f$ is continuous, there exists a point $z^{k+1} \in \intr C$ that can be reached by a sufficiently small step along the direction $d^k$. This ensures that both the post-smoothing step and the two-level BPGD scheme, as discribed in \cref{alg:Two-level-BPGD}, are well-defined.

It remains to be shown that the descent direction $d^k$, computed via coarse correction, indeed constitutes a descent direction for the fine-level objective $f$.
\begin{proposition}[$d^k$ is a descent direction] \label{prop:dk-is-a-descent-direction}
    Let $f_H$ be convex and $L_H$-smooth relative to a convex function $\vphi_H$. Suppose that in the 
    $k$-th iteration of \cref{alg:Two-level-BPGD}, the coarse update direction satisfies $x_H^{m,k}-x_H^{0,k} \in \rg(R)$. Then, the coarse correction direction $d^k := P(x_H^{m,k} - x_H^{0,k})$ is a descent direction of $f$ at $x^k$. 
\end{proposition}
\begin{proof}
    Formally, we show that
    \begin{equation}
        \la \df(x^k), d^k \ra < 0.
    \end{equation}
    This is a straightforward computation. We start by employing the first-order coherence condition.
    \begin{align}
        \la \df(x^k), P(x_H^{m,k}-x_H^{0,k}) \ra &= \la \nabla \psi^k(x^{0,k}_H), x^{m,k}_H - x^{0,k}_H \ra \\
        &\leq \psi^k(x^{m,k}_H) - \psi^k(x^{0,k}_H) \\ &\leq -\sum_{i = 1}^m \frac{D_{\vphi_H}(x^{i-1,k}, x^{i,k})}{\tau_H} < 0. \label{eq:coarse-sufficient-descent-Armijo}
    \end{align}
    The first inequality is a consequence of the nonnegativity of $D_{\psi^k}(x^{m,k}_H,x^{0,k}_H)$, ensured by the convexity of $f_H$. The second follows from successive applications of the sufficient descent property of BPGD (\cref{lm:sufficient-descent-BPGD}). The strict inequality is guaranteed by $x_H^{m,k} \neq x_H^{0,k}$ and the strict convexity of $\vphi_H$. 
\end{proof}
We conclude the discussion on the feasibility of the two-level variant of ML-BPGD by showing that it satisfies the following fixed-point property.
\begin{lemma}[Fixed point property of \cref{alg:Two-level-BPGD}] \label{lm:fixed_point_property-MLBPGD}
    If $x^k$ a critical point of $f$, then \cref{alg:Two-level-BPGD} yields $x^{k+1} = x^k$.
\end{lemma}
\begin{proof}
    By \cref{lm:transferability-of-criticality}, $x_H^{0,k}$ is a critical point of $\psi^k$ since $x^k$ is a critical point of $f$. Successive use of \cref{lm: fixed point property of BPGD} yields $x^{m,k}_H = x^{0,k}_H$. Thus, the direction $d^k$ vanishes and the equality $x^{k+1} = x^k$ holds by the same \cref{lm: fixed point property of BPGD}.
\end{proof}

\subsubsection{Convergence}\label{sec:2l-BPGD-convergence}
Using an Armijo-based backtracking approach to obtain $z^k$ allows us to quantify the descent of a coarse step.
Given a descent direction $d$ of $f$ at a point $x$ and constants $\bar \alpha > 0$, $\beta, \sigma \in (0,1)$, the Armijo line search outputs $\alpha := \beta^m \bar \alpha$, where $m$ is the smallest nonnegative integer satisfying
\begin{equation} \label{eq:Armijo-line-search}
    f(x + \alpha d ) \leq f(x) + \sigma \alpha \la \df(x), d \ra.
\end{equation}
Taking $x = x^{k-1}$ implies
\begin{align}
    f(z^k) &\leq f(x^{k-1}) + \sigma \alpha_{k-1} \la \df(x^{k-1}), d_{k-1} \ra \\
    &\overset{\eqref{eq:coarse-sufficient-descent-Armijo}}{\leq }f(x^{k-1}) - \sigma \alpha_{k-1} \sum_{i = 1}^m \frac{D_{\vphi_H}(x^{i-1,k-1}, x^{i,k-1})}{\tau_H}. \label{eq:Armijo-descent}
\end{align}
This allows us to relate the descent of the coarse correction to the sufficient descent obtained by solving the coarse problem. We now state and prove our main result.

\begin{theorem}\label{thm:2l-convergence}
Let $f$ satisfy \cref{ass:solveability,ass:PL-with-L-dependence} with constants $L, \eta > 0$, and scaling function $\theta$, and let $f_H$ satisfy \cref{assump:solvability-coarse-2l} with $L_H > 0$. For constant step sizes $\tau \in (0, L^{-1}]$ and $\tau_H \in (0, L^{-1}_H]$ the following holds. Employing an Armijo backtracking line search \eqref{eq:Armijo-line-search} for the coarse corrections to obtain $z^{k+1}$, the function values of the iterates $\{x^k\}_{k \in \N}$ of \cref{alg:Two-level-BPGD} converge. More precisely,
\begin{equation}
    f(x^k) - f^{\min} \leq \left(1- r \right)^k (f(x^0) - f^{\min}) - \sum_{j = 0}^{k-1} (1-r)^{k-j} \rho^j
\end{equation}
with 
\begin{equation}
    \rho^j = 
    \begin{cases}
        \sigma \alpha_j \sum_{i = 1}^m \frac{D_{\vphi_H}(x^{i-1,j-1}, x^{i,j-1})}{\tau_H}, &j \text{ triggers a coarse correction} \\
        0, &\text{otherwise}
    \end{cases}
\end{equation}
and constant $r := \frac{\theta(\tau) \eta}{\tau}$.
\end{theorem}
\begin{proof}
Let the iterate $k \in \N$ satisfy a coarse correction condition, then
    \begin{align} \label{eq:2l-sufficient-descent}
        f(x^k) &\leq f(z^k) - \frac{1}{\tau} D_\vphi(z^k, x^k) \\ &\leq f(z^k) - \frac{1}{\tau} \theta(\tau) \eta (f(z^k) - f^{\min}).
    \end{align}
    Subtracting $f^{\min}$ from both sides yields
    \begin{equation} \label{eq:convergence-coarse-step}
        f(x^k) - f^{\min} \leq (1-r) (f(z^{k}) - f^{\min}) \overset{\eqref{eq:Armijo-descent}}{\leq} (1-r) (f(x^{k-1}) - f^{\min}) - (1-r)\rho_H^{k-1}.
    \end{equation}
    Now, when $k$ does not trigger a coarse correction, then $f(x^k) \leq f(x^{k-1}) - r(f(x^k) - f(x^{\min}))$ holds analogously to  \eqref{eq:2l-sufficient-descent}. 
    Applying this relation and \eqref{eq:convergence-coarse-step} recursively proves the result.
\end{proof}
\subsection{Multilevel BPGD}\label{sec:MLBPGD}
We extend the multilevel BPGD to multiple coarse level via a recursive approach. This section only expands the notation, setup and assumptions required for the well-definedness of a ML-BPGD beyond one coarse level. We remark at the end of the section on how the theoretical results generalize to this extended setting.

Let $\mc{L}$ denote the coarsest level and $\{n_\ell\}_{\ell \in [\mc{L}]} \subset \N$ decreasing dimensions, with $n =: n_0 \gg n_1 \gg n_2 \gg \dots \gg n_{\mc{L}}$. Analogous to the special case introduced in \Cref{sec:2l-BPGD}, we assume we have access to $\mc{L}$  coarse convex versions $\{f_\ell:\R^{n_\ell} \to (-\infty, \infty]\}_{\ell \in [\mc{L}]}$ of the fine objective $f$. Restriction maps $R_\ell: \R^{n_\ell} \to \R^{n_{\ell + 1}}$ and their corresponding prolongations $P_{\ell} = R_\ell{^\top}$ are also provided. The multilevel V-cycle variant of the algorithm is summarized in \cref{alg:ML-BPGD} with
the $\ell$-th coarse model at the $k$-th iteration given by
\begin{equation} \label{eq:coarse-model-ML}
    \psi_{\ell}^k(x) := f_{\ell}(x) + \la v_{\ell}^k, x - x_{\ell}^{0,k} \ra, \qquad v_{\ell}^k = R_{\ell}\nabla \psi_{\ell-1}^k(x_{\ell - 1}^{m_{\ell-1},k}) - \df_{\ell}(x_\ell^{0,k}),
\end{equation}
which satisfies the first-order coherence condition
\begin{equation}
    \nabla \psi_{\ell}^k(x_\ell^{0,k}) =  R_{\ell}\nabla \psi_{\ell-1}^k(x_{\ell - 1}^{m_{\ell-1},k})
\end{equation}
on each level.

\begin{algorithm}[h]
\DontPrintSemicolon
\textbf{initialization: $x^0 \in \intr C$, $k \in \N$}\;
\Repeat{a stopping rule is met} {
  \For{$\ell = 0, \dots, \mc{L}-1$}{
  \If{condition to use coarse model is satisfied at $x^{m_\ell,k}_{\ell}$} {
    $x_{\ell+1}^{0,k} = R_\ell x^{m_\ell,k}_{\ell}$ \;
    \For{$i \in [1, \cdots, m_{\ell+1}]$} {
    $x_{\ell+1}^{i,k} = \argmin_{u \in C_{\ell+1}^k} \Phi(u; x_{\ell+1}^{i-1,k}, \tau_{\ell+1}, \psi_{{\ell+1}}^k, \vphi_{{\ell+1}}^k)$
    }
  }}\Else{
  $\mc{L} = \ell$
  }
  \If{$\mc{L} > 0$}{
  $x_\mc{L}^{k+1} = x_\mc{L}^{m_{\mc{L}},k}$ \;
  \For{$\ell = \mc{L}-1, \dots, 0$}{
    $d_{\ell}^k = P_\ell (x_{\ell+1}^{k+1} - x_{\ell+1}^{0,k})$ \;
    Find $\alpha_{\ell}^k > 0$ such that $\psi_{\ell}^k(x^{m_\ell,k}_{\ell} + \alpha_{\ell}^k d_{\ell}^k) \leq \psi_{\ell}^k(x^{m_\ell,k}_{\ell})$ \;
    $z_\ell^{k+1} = x^{m_\ell,k}_\ell + \alpha_{\ell}^k d_{\ell}^k$ \;
    $x^{k+1}_\ell = \argmin_{u \in C_\ell^k} \Phi(u; z_{\ell}^{k+1}, \tau_\ell, \psi_{\ell}^k, \vphi_\ell^k)$
  }
  }
  \Else{
  $x^{k+1} = \argmin_{u \in C} \Phi(u; x^k, \tau, f, \vphi)$
  }
  Increment $k \leftarrow k+1$
}
\caption{Multilevel BPGD}\label{alg:ML-BPGD}
\end{algorithm}

We check at the $(\ell-1)$-th level if $x_{\ell-1}^{m_{\ell-1}, k}$, the current best approximation of the coarse model $\psi_{\ell-1}^k$, satisfies the following coarse correction condition
\begin{condition}
    We ease notation by defining
    \begin{equation}
        \chi_\ell^{i,k} = \abs{ \min_{x_\ell^{i,k} + d \in C_\ell^k} \la \nabla \psi_\ell^k(x_\ell^{i,k}), d \ra}.
    \end{equation}
    The iterate $x_{\ell-1}^{m_{\ell-1},k}$ triggers a coarse correction step to be computed on the $\ell$-th level if the following holds
    \begin{equation}
        \chi_\ell^{0,k} \geq \kappa_\ell \chi_{\ell-1}^{m_{\ell-1},k}, \quad \chi_{\ell-1}^{m_{\ell-1},k} \geq \epsilon_\ell
    \end{equation}
    for $\kappa_\ell, \epsilon_\ell \in (0,1)$, and
\begin{equation}
    D_\vphi(x^{m_{\ell-1},k}_{\ell-1} ,\widetilde{x}_{\ell-1}) \geq \epsilon_x,
\end{equation}
where $\epsilon_x > 0$, and $\widetilde{x}_{\ell-1}$ is the last iterate that triggered a multilevel step on the $(\ell-1)$-th level.
\end{condition}
If so, we solve, or rather approximate using $m_{\ell}$ BPGD iterates, the coarse convex problem
\begin{equation}
    \min_{x \in C_{\ell}^k} \psi_{\ell}^k(x), \quad C_{\ell}^k \subseteq \{w \in \R^{n_\ell}: x_{\ell-1}^{m_{\ell-1}, k} + P_{\ell-1}(w - x_{\ell}^{0,k}) \in C_{\ell-1}^k \}.
\end{equation}
\begin{remark}[Notation] To have a consistent notation, we identify the $0$-th level with the finest level, obtaining the identities $\psi_{0}^k \equiv f$ for all $k \in \N$, with its constraint set $C_0^k = C$, and $x_0^{m_0,k} = x_{0}^{i,k}=x^k$ for all $i \in \N$.
    
\end{remark}
The considerations for the well-definedness and solvability of the two-level BPGD extend to the multilevel case. For each level, we choose a strictly convex reference function $\vphi_\ell$ to write the multilevel assumptions
\begin{assumptionp}{D}\label{assump:solvability-coarse-ml}
    \begin{enumerate}
    \item[~] 
    \item $f_\ell$ is $L_\ell$-smooth relative to $\vphi_\ell$ on $C_\ell^k$ for all $k$
    \item the subproblem
    \begin{equation}
        \argmin_{u \in C_\ell^k} \Phi(u; x, \tau_\ell, \psi_{\ell}^k, \vphi_\ell)
    \end{equation}
    always has a solution on $C_\ell^k$ that is a singleton and is efficiently computable.
\end{enumerate}
\end{assumptionp}
Note that one can choose $\vphi_\ell$ to vary at each iteration to best match the constraints $C_\ell^k$, which change according to $k \in \N$, as illustrated in \cref{alg:ML-BPGD}. To simplify the already convoluted notation, we only provide explicit analysis for the case of a uniform $\vphi_\ell$ over all iterations. This yields no significant changes in the convergence theory or behavior of the algorithm apart from an additional superscript.

\subsubsection{Well-defined algorithm and convergence in the multilevel case}

The results established in \Cref{sec:2l-BPGD-well-defined,sec:2l-BPGD-convergence} for the two-level setting extend naturally to the multilevel case. Instead of repeating proofs, we briefly outline the generalization of key properties.

\paragraph{Fixed point property} 
The fixed point property from \cref{lm:fixed_point_property-MLBPGD} extends directly. Specifically, all points $x_{\ell}^{0,k}$ remain critical for $\psi_{\ell}^k$ whenever $x^k$ is a critical point of $f$, see \cref{lm:transferability-of-criticality}. The fixed point property of \cref{alg:ML-BPGD} thus follows as a natural consequence of the fixed point property of the BPGD (\cref{lm: fixed point property of BPGD}).

\paragraph{Descent direction} 
It remains true that $d_0^k$ is a descent direction for $f$ at $x^k$, see \cref{prop:dk-is-a-descent-direction}. The key observation is that the descent property is preserved recursively across levels, starting with 
\begin{equation}
    \la \nabla \psi^k_{\mc{L}-1}(x_{\mc{L}-1. k}^{m_{\mc{L}-1}}), d_{\mc{L}-1}^k \ra \leq - \frac{1}{\tau_{\mc{L}}}\sum_{i = 1}^{m_\mc{L}} D_{\vphi_{\mc{L}}}(x_{\mc{L}}^{i-1,k}, x_{\mc{L}}^{i,k}). 
\end{equation}
The hierarchical structure allows us to propagate the descent property across  levels, provided that the Armijo backtracking condition is enforced at all levels. The same arguments as in the two-level case then yield:
\begin{align}
    \la \df(x^k), d_{0}^{k} \ra &\leq \psi_{1}^k(x_1^{m_1,k}) - \psi_{1}^k(x_1^{0,k}) + \sigma \alpha_{1}^k \la \nabla \psi_{1,k}(x_1^{m_{1}}), d_{1}^k \ra \\
    & \leq \cdots \leq - \sum_{\ell = 1}^{\mc{L}} \nu_\ell \sum_{i = 1}^{m_\ell} \frac{1}{\tau_\ell} D_{\vphi_\ell}(x_\ell^{i-1,k}, x_\ell^{i,k}) < 0,
\end{align}
where $\nu_\ell = \prod_{j = 1}^{\ell - 1} \sigma \alpha_{j}^k$ holds for $x_{\ell}^{k+1} - x_\ell^{0,k} \in \rg(R_\ell)$. This confirms that the descent property holds in the multilevel case.

\paragraph{Convergence} 
The convergence result in \cref{thm:2l-convergence} extends to the multilevel case under the same assumptions. If $f$ satisfies \cref{ass:solveability,ass:PL-with-L-dependence} with constants $L, \eta > 0$ and function $\theta$, and the coarse functions satisfy \cref{assump:solvability-coarse-ml} with $L_\ell > 0$, then the function values of the iterates $\{x^k\}_{k \in \N}$ of \cref{alg:ML-BPGD} converge. More precisely:
\begin{equation}
    f(x^k) - f^{\min} \leq (1-r)^k f^0 - f^{\min} - \sum_{t = 0}^{k-1}(1-r)^{k-t}\rho^t
\end{equation}
with 
\begin{equation}
    \rho^t = 
    \begin{cases}
       \sum_{\ell = 1}^{\mc{L}} \prod_{j = 1}^{\ell - 1} \sigma \alpha_{j,t} \sum_{i = 1}^{m_\ell} \frac{1}{\tau_\ell} D_{\vphi_\ell}(x_\ell^{i-1,t}, x_\ell^{i,t}), &t \text{ triggers a coarse correction} \\
        0, &\text{otherwise}
    \end{cases}
\end{equation}
and constant $r := \frac{\theta(\tau) \eta}{\tau}$. The proof follows the same reasoning as in the two-level case, but extends naturally across multiple levels.

Thus, extending to the multilevel setting introduces no fundamental difficulties, and all key properties remain valid.

\subsection{Examples of constructing coarse constraints}
In this section, we focus on the consistency of feasibility, cf. \Cref{sec:2l-BPGD-coarse-model}, and provide explicit examples of how to handle the constraints in our multilevel framework. The following special cases are relevant for the numerical experiments in \Cref{sec:NumExperiments}.
To simplify notation in this section, we use the same symbols $P$ and $R$ for the level-dependent prolongation and restriction operators.

\subsubsection{Separable linear constraints} \label{appdx:constraints_separable_linear}
A key result for adaptable separable linear bounds, originally proposed in \cite{Gelman1990} and later generalized in \cite{Gratton:2008}, is employed in our examples. For $u_0, l_0 \in \R^{n_0}$, with $-\infty \leq l_0 \leq u_0 \leq \infty$ we define the adapted coarse constraints for the $\ell$-th level with initial value $x_\ell = Rx_{\ell-1}$ using the $\ell_\infty$ recursive update 
\begin{equation} \label{eq: recursive lower bound}
    \{l_\ell\}_j = \{x_\ell\}_j + \frac{1}{\norm{P}_\infty} \max_{t \in [n_{\ell -1}]} \begin{cases}
        \{l_{\ell - 1} - x_{\ell - 1}\}_t, & \quad P_{tj} > 0 \\
        \{x_{\ell - 1} - u_{\ell - 1}\}_t, & \quad P_{tj} < 0
    \end{cases}
\end{equation}
and
\begin{equation} \label{eq: recursive upper bound}
    \{u_\ell\}_j = \{x_\ell\}_j + \frac{1}{\norm{P}_\infty} \min_{t \in [n_{\ell -1}]} \begin{cases}
        \{u_{\ell - 1} - x_{\ell - 1}\}_t, & \quad P_{tj} > 0 \\
        \{x_{\ell - 1} - l_{\ell - 1}\}_t, & \quad P_{tj} < 0.
    \end{cases}
\end{equation}

The recursive definitions rely on the prolongation operators $P$, which are crucial in ensuring the feasibility of the iterates obtained from the coarse correction steps. For a detailed discussion of our choice of such matrices $P$, see \Cref{sec:NumExperiments,sec:2l-BPGD-transfer-ops}. The feasibility of the iterates is guaranteed by the following lemma.
\begin{lemma}[\cite{Gratton:2008}, Lemma 4.3]\label{lm: gratton}
    Let $C := [l_0, u_0]$ and define $C_{\ell}^k := [l_\ell^k, u_\ell^k]$ recursively according to definitions \eqref{eq: recursive lower bound}, \eqref{eq: recursive upper bound} for the points $x_\ell^{0,k}=Rx_{\ell-1}^{m_{\ell-1},k}$. This enforces the inclusion
    \begin{equation}
        x_{\ell-1}^{m_{\ell-1}, k} + P(w - x_{\ell}^{0,k}) \in C_{\ell-1}^k \quad \text{for all} \; w \in C_{\ell}^k
    \end{equation}
    for all levels $\ell$ and all iterates $k$.
\end{lemma}
\begin{remark}
    Note, $u = \infty$ corresponds to nonnegative constraints. In this case, we only need to adapt the lower bounds to obtain consistent constraints. The case of negative constraints $l = - \infty$ follows the same argument.
\end{remark}
\subsubsection{Nonseparable linear constraints: the simplex}
Accommodating simplex constraints to the above setting by adding an equality constraint involving all the variables, that is
\begin{equation}
    \sum_{i = 1}^n x_i = S,
\end{equation}
turns the constraints inseparable. For $l \in \R^n$ and scalar $S$, let \begin{equation}
    \Delta^{n}(l,S) := \{ x \in \R^{n}: \sum_{i = 1}^{n} x_i = S, x_i \geq l \}
\end{equation}
denote the scaled and translated $n$-dimensional probability simplex. The following proposition describes how we can choose the coarse constraints to prolong back to the standard probability simplex under the right conditions.
\begin{proposition} \label{prop:criticality-consistency-simplex}
    Let $C = \Delta^n$. Set $S_\ell^k:= \sum_{i = 1}^{n_\ell} \{x^{0,k}_\ell\}_i$, and define $C_\ell^k := \Delta^{n_\ell}(l^k_\ell, S_\ell^k)$ recursively with \eqref{eq: recursive lower bound} and $x_\ell^{0,k}=Rx_{\ell-1}^{m_{\ell-1},k}$. If the prolongation operator is an interpolator satisfying the partition-of-unity property, then 
    \begin{equation}
        x_{\ell-1}^{m_{\ell-1}, k} + P(w - x_{\ell}^{0,k}) \in C_{\ell-1}^k \quad \text{for all} \; w \in C_{\ell}^k
    \end{equation}
    for all levels $\ell$ and all iterates $k$.
\end{proposition}
\begin{proof}
    We only need to check the equality constraint, since \cref{lm: gratton} ensures that the the lower bounds satisfy the needed property. Let $w \in C_\ell^k$. Setting $z = x_{\ell-1}^{m_{\ell-1}, k} + P(w - x_{\ell}^{0,k})$ for ease of notation, it suffices to show
    \begin{equation} \label{eq:simplex-sonstraints-ZE}
        \sum_{i = 1}^{n_\ell} w_i - \{x_\ell^{0,k}\}_i = 0,
    \end{equation}
    since $P$ is sum-preserving, which implies $\la  P(w - x_{\ell}^{0,k}), \eins \ra = 0$ and thus $\la z, \eins \ra = \la x_{\ell-1}^{m_{\ell-1}, k}, \eins \ra \allowbreak = S_{\ell - 1}^k$. Since \eqref{eq:simplex-sonstraints-ZE} is a tautological consequence of the construction of $C_\ell^k,$  the statement follows.  
\end{proof}

\section{Numerical experiments}\label{sec:NumExperiments}
The goal of this chapter is to demonstrate the advantages of the proposed ML-BPGD framework for different image reconstruction purposes: Poisson-noisy deconvolution, \Cref{sec:deconvolution}, tomographic reconstruction, \Cref{sec:tomographic_reconstruction}, and its optimal design, \Cref{sec:d-optimal-design}. Code and examples illustrating these numerical results are available here\footnote{\url{https://github.com/yaraelshiaty/multigrid}}. 

\begin{figure}[h!]
    \centering
    \begin{subfigure}{0.32\textwidth}
        \centering
        \includegraphics[width=0.95\linewidth]{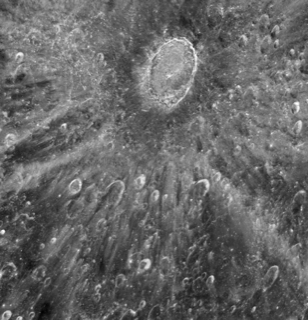}
    \end{subfigure}
    \hfill
    \begin{subfigure}{0.32\textwidth}
        \centering
        \includegraphics[width=\linewidth]{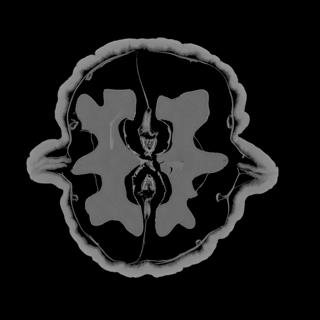}
    \end{subfigure}
    \hfill
    \begin{subfigure}{0.32\textwidth}
        \centering
        \includegraphics[width=\linewidth]{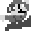}
    \end{subfigure}
    
    \caption{\textbf{The reference images for the three experiments.} \texttt{Left}: Crater Tycho on the Moon, taken by the Hubble Space Telescope, \url{https://science.nasa.gov/image-detail/tycho-crater/}, for Poisson-noisy deconvolution. \texttt{Center}: Walnut Phantom, \cite{tomographicxraydatawalnut}, for tomographic reconstruction. \texttt{Right}: Jumping Mario, for D-optimal design in tomography.}
    \label{fig:reference-images}
\end{figure}

We present the problems in their fine formulation, while highlighting the geometry that underlies them. Then, we contextualize them within the multilevel framework and 
derive suitable coarse models using established trust-region methods. Here, we detail the choices for the coarse geometry and discuss the feasibility of our ML-BPGD framework. 
The image sizes ($511^2, 1023^2$) used in the experiments are representative of typical reconstruction tasks and already sufficient to reveal the convergence benefits of the multilevel approach. For the third experiment, which is based on the Fisher-information objective, smaller images are employed since the underlying computations become numerically unstable and increasingly expensive as resolution grows, making a multilevel treatment particularly valuable even at modest sizes.

\begin{remark}
    Our multilevel framework presented in \cref{alg:ML-BPGD} has many technical hyperparameters to tune: the number of levels and 
    the number of iterations within each level, the transfer operators, the coarse correction condition and its parameters. In the experiments presented in this paper, we have selected specific values for each of these parameters. These choices were made to match the needs of each of the problems, but we do not claim or demonstrate that they represent optimal configurations.
\end{remark}

\subsection{Deconvolution} \label{sec:deconvolution}
In astronomical image processing, the goal is to recover the true image of celestial objects from Poisson-distributed noisy measurements, often representing photon counts distorted by the telescope's point spread function (PSF). The relationship between the data and the unknown image is modeled with a nonnegative linear operator $A \in \R^{m \times n}_{+} $ (with nonzero rows), and the measurement vector $b \in \R^m_{++}$ is subject to Poisson noise. A natural proximity measure for this problem is the Kullback-Leibler ($\KL$) divergence, which, when minimized, is equivalent to maximizing the Poisson log-likelihood. Therefore, we consider the objective:
\begin{equation}\label{eq:objective-astro-klbAx}
    \min_{x \in \R^n_+} \KL(b, Ax) = \la b, \ln \frac{b}{Ax} \ra - \la \eins , b - Ax \ra.
\end{equation}
which aims to find the $I$-projection \cite{nielsen2018informationprojection} of $b$ onto the nonnegative orthant.
While \eqref{eq:objective-astro-klbAx} is convex, it lacks a globally Lipschitz continuous gradient. However, \eqref{eq:objective-astro-klbAx} is $\norm{b}_1$-smooth relative to the log-barrier function $\vphi(x) = - \la \eins, \ln x \ra$, compare \cref{lm:klbAx_rel_smth_to_log_barrier}. 

\paragraph{Experimental setup} We consider the Crater Tycho on the Moon image (\cref{fig:reference-images}, left), taken by the Hubble Space Telescope, scaled to the size $512 \times 512$ and 
blurred by a PSF kernel, with Poisson noise added to the resulting image $b$.
. We consider $4$ different scenarios corresponding to different combinations of the size of the Gaussian blur PSF and the level of Poisson noise, see \cref{tbl:blur_vs_noise}. We initialize all experiments with $x^0 = 0.5 \cdot \eins_n$.

\begin{table}[h]
\centering
\begin{tabular}{lcc}
\toprule
& $\lambda (\text{noise}) = 1000$ & $\lambda(\text{noise}) = 15$ \\
\midrule
dim(PSF) = 15, $\sigma(\text{PSF}) = 1.5$ & low blur, low noise & low blur, high noise \\
dim(PSF) = 27, $\sigma(\text{PSF}) = 5$ & high blur, low noise & high blur, high noise \\
\bottomrule
\end{tabular}
\caption{Four configurations of Gaussian blur convolution with multiplicative Poisson noise, $b \sim \frac{1}{\lambda} \Poi(\lambda A(\text{input})\protect\footnotemark[1]$. Image reconstructions and decay of objective functions are presented in \cref{fig:images-deconvolution,fig:CPU-deconvolution}, respectively. The dimension of a PSF is the width of its kernel, while $\sigma$, the standard deviation, controls the spread of that blur within the kernel. They construct the blur matrix $A$.}
\label{tbl:blur_vs_noise}
\end{table}

\footnotetext{
The elements of the measurement vector $b$ follow a Poisson distribution with $\mathbb{E}[b_i] = A_i(\text{input})$ and
$\Var[b_i] = \frac{A_i(\text{input})}{\lambda}$. Smaller values of $\lambda$ imply higher variance and thus more noise. }

\begin{figure}[h]
\centering

\begin{subfigure}{0.45\textwidth}
    \centering
    \caption*{\small \texttt{low blur, low noise}}
    \vspace{-0.3em}
    \includegraphics[width=1.\linewidth]{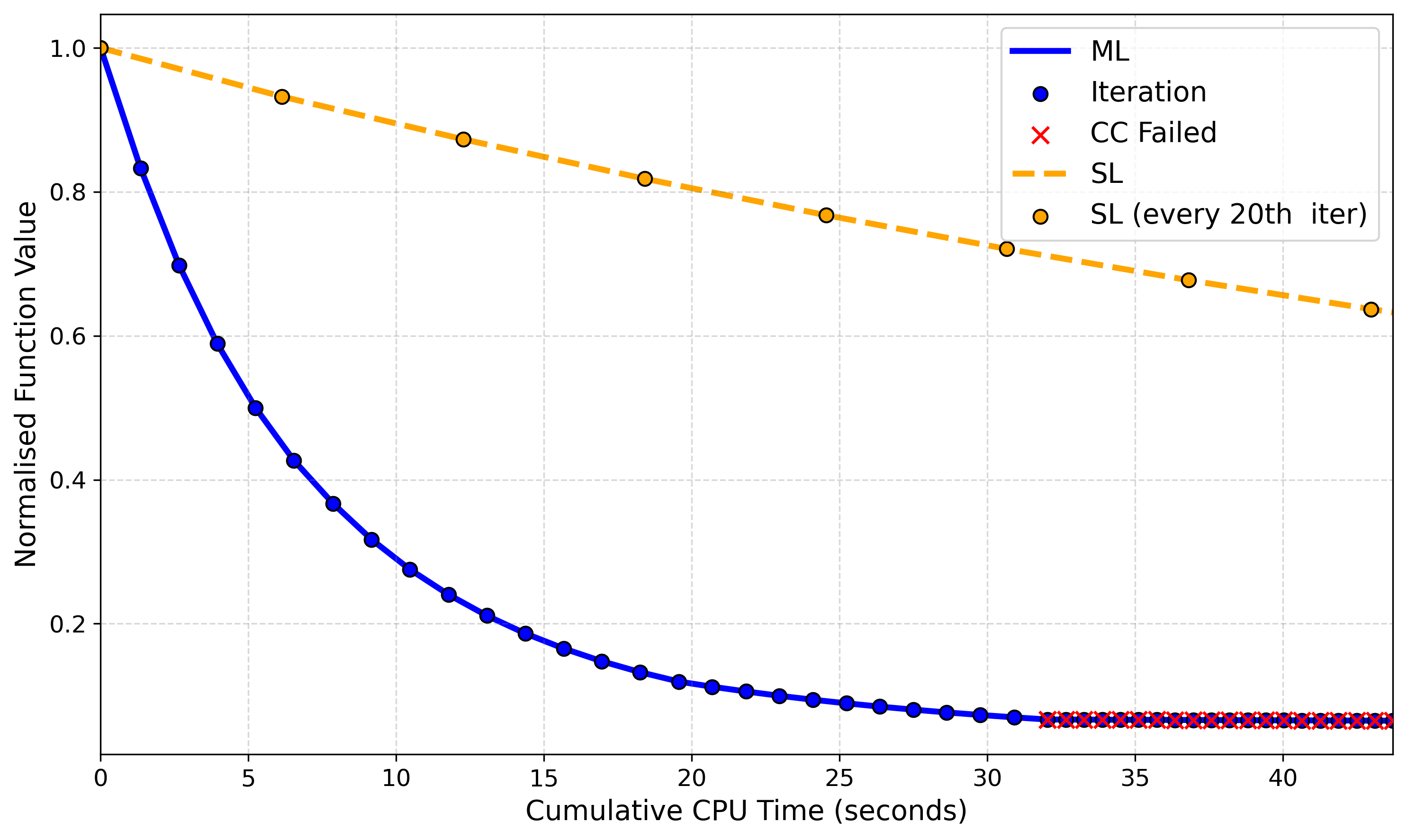}
\end{subfigure}
\hspace{.5em}
\begin{subfigure}{0.45\textwidth}
    \centering
    \caption*{\small \texttt{low blur, high noise}}
    \vspace{-0.3em}
    \includegraphics[width=1.\linewidth]{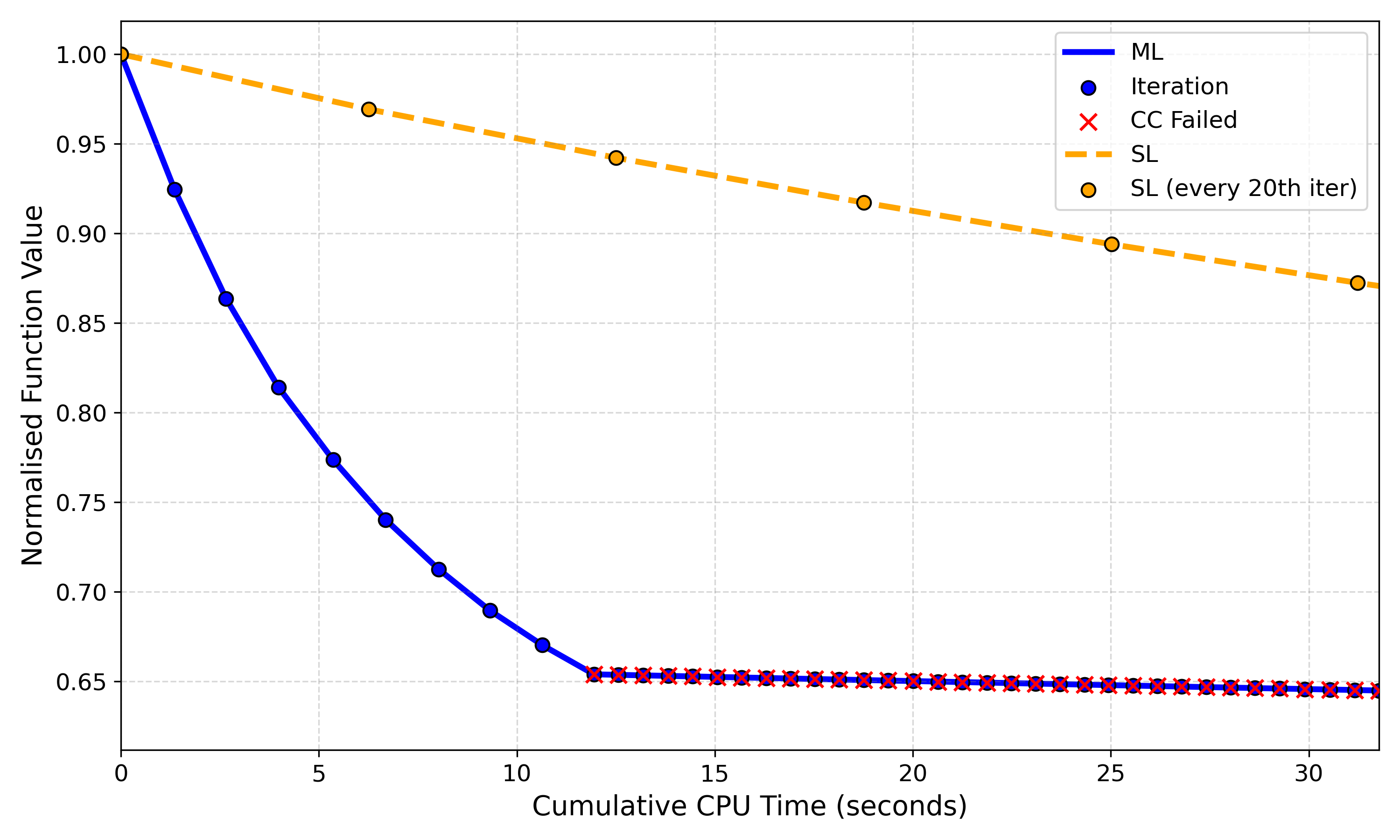}
\end{subfigure}

\vspace{1em}

\begin{subfigure}{0.45\textwidth}
    \centering
    \caption*{\small \texttt{high blur, low noise}}
    \vspace{-0.3em}
    \includegraphics[width=1.\linewidth]{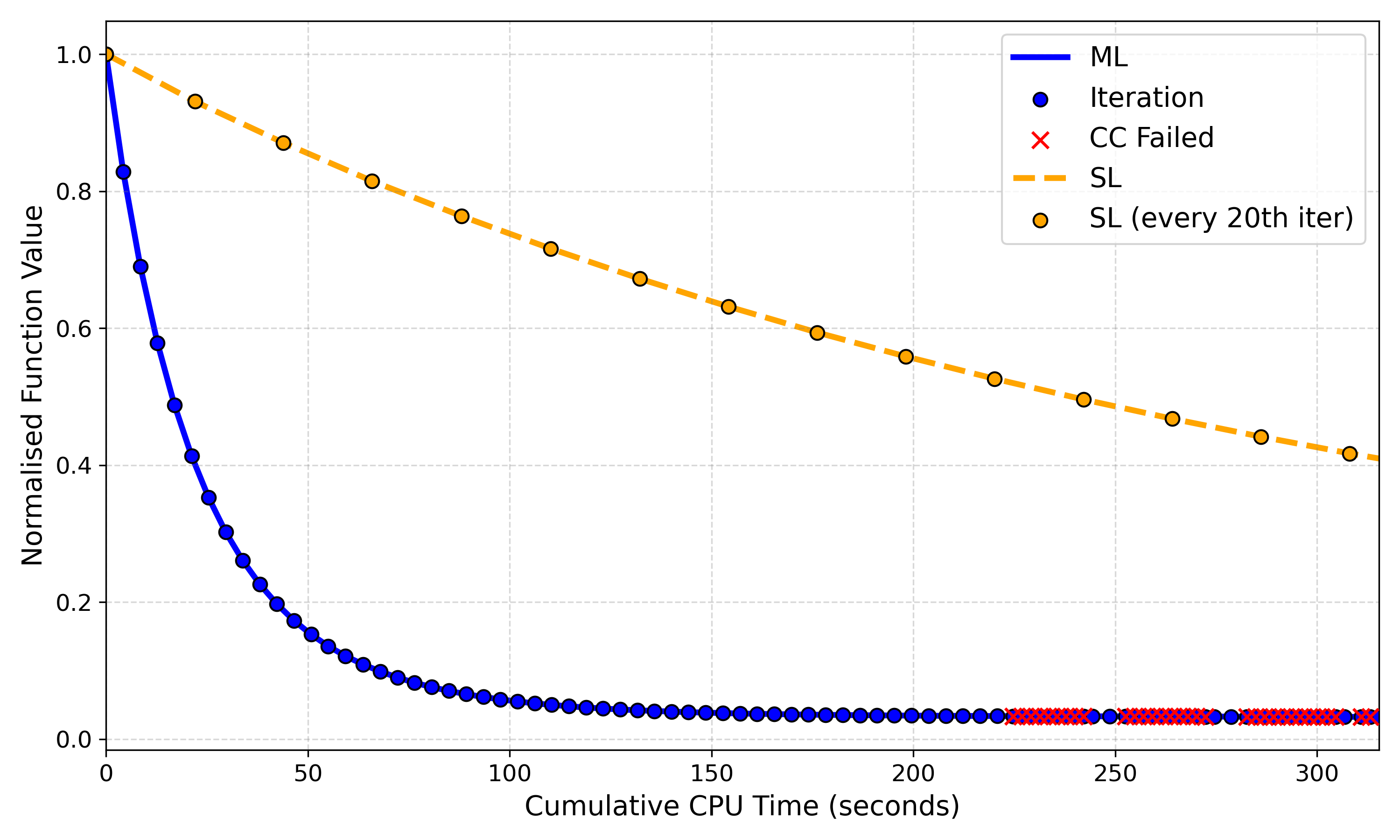}
\end{subfigure}
\hspace{1em}
\begin{subfigure}{0.45\textwidth}
    \centering
    \caption*{\small \texttt{high blur, high noise}}
    \vspace{-0.3em}
    \includegraphics[width=1.\linewidth]{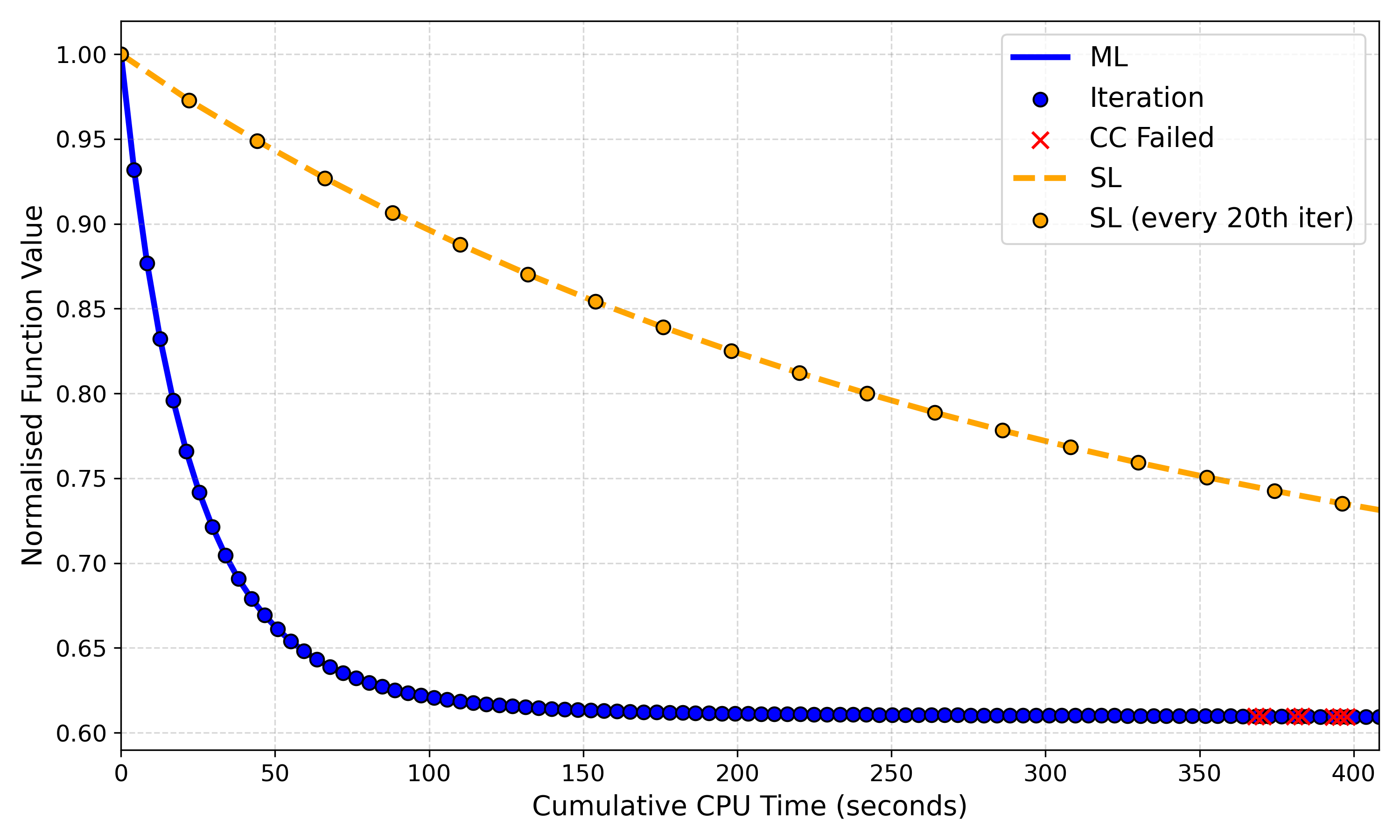}
\end{subfigure}

\caption{\textbf{Normalized function value vs CPU time (in seconds)}. Deblurring performance across the various blur and noise conditions specified in \cref{tbl:blur_vs_noise}, using the Tycho Crater image. \texttt{Yellow (SL):} Single-level BPGD, with markers shown every 20 iterations. \texttt{Blue (ML):} ML-BPGD with two coarse levels; markers are shown every iteration. Violations of the coarse correction condition at the finest level are indicated by red 'x' markers in the plot. Our ML-BPGD far outperforms the single-level variant across all specified blur and noise variants.}
\label{fig:CPU-deconvolution}
\end{figure}

\paragraph{Multilevel structure} We use \cref{alg:ML-BPGD} with a total of $3$ levels, with coarse grid sizes $255 \times 255$ and $123 \times 123$; one iteration performed on the finest level, and $10$ iterations each for the coarse ones. Images are transferred between the levels using the bilinear interpolator kernel $K_{2{\text{D}}}$, see \eqref{eq:2D-kernel} , whereas an Armijo line search is employed for the coarse corrections.
The size of the Gaussian blur kernel, 
its standard deviation and the expected Poisson noise level do not change across levels.
Coarse correction steps use the constants $\kappa = 0.49$ and $\epsilon = 1e-3$ for checking the coarse correction criteria. We use the unconstrained version to simplify computations, cf. the discussion in \Cref{sec:2l-BPGD-CCC} for details. The evolution of the function value to CPU time for the four different setups is displayed in \cref{fig:CPU-deconvolution}. Effectively, the objective function in ML-BPGD decreases more rapidly than in BPGD, achieving comparable reductions approximately 40 iterations earlier. The deblurred and denoised images obtained after 60 iterations are shown in \cref{fig:images-deconvolution}.
\paragraph{Coarse problem construction}
The coarse minimization problems are given by
\begin{equation} \label{eq:coarse-model-KLbAx-astro}
    \min_{x \in C_\ell^k} \psi_{\ell}^k(x), \quad \psi_{\ell}^k \text{ as in } \eqref{eq:coarse-model-ML}, \quad f_\ell(x) = \KL(b_\ell, A_\ell x)    
\end{equation}
with the Gaussian blur $A_\ell$ and the noisy and blurred image $b_\ell$, for $\ell = 1,2$, across all iterations $k \in \N$.
The constraint sets $C_\ell^k = [l_\ell^k, \infty)$ are defined by adapting the upper bound as discussed in detail in \cref{appdx:constraints_separable_linear}. We initialize with $l_0^k = 0$ for all $k$ and recursively compute the updated bounds by \eqref{eq: recursive lower bound}. Note that all entries of the bilinear interpolator $P$ are positive and its row sums are normalized.
\cref{prop:log-barrier-orthant-to-box} shows that the coarse objectives are $\norm{b_\ell}_1$-smooth relative to the 
adapted log-barrier function
\begin{equation}
    \vphi(x) = - \sum_{i = 1}^{n_\ell} \ln (x_i - \{l_\ell^k\}_i).
\end{equation}
The efficient computation of the BPGD updates for \eqref{eq:objective-astro-klbAx} and \eqref{eq:coarse-model-KLbAx-astro} is thoroughly discussed in \Cref{sec:solving_BPGD_iterates_log_barrier}. We use constant step sizes given by the inverse of the relative smoothness constants $\tau_\ell = \norm{b_\ell}_1^{-1}$ for all levels.

\begin{figure}[ht]
\centering
\resizebox{0.7\textwidth}{!}{%
\begin{tabular}{c c c c}
    & \texttt{Blurred/Noisy Image} & \texttt{SL Reconstruction} & \texttt{ML Reconstruction} \\
    
    \raisebox{1\height}{\rotatebox[origin=c]{90}{\tiny \texttt{low blur, low noise}}} &
    \includegraphics[width=0.25\textwidth]{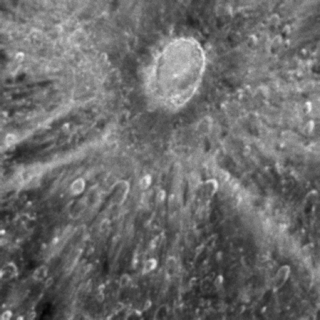} &
    \includegraphics[width=0.25\textwidth]{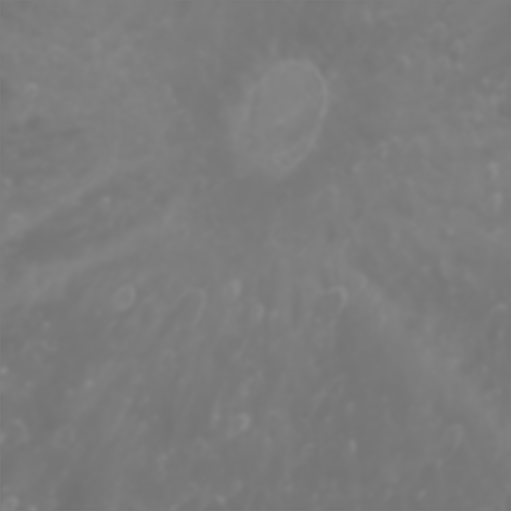} &
    \includegraphics[width=0.25\textwidth]{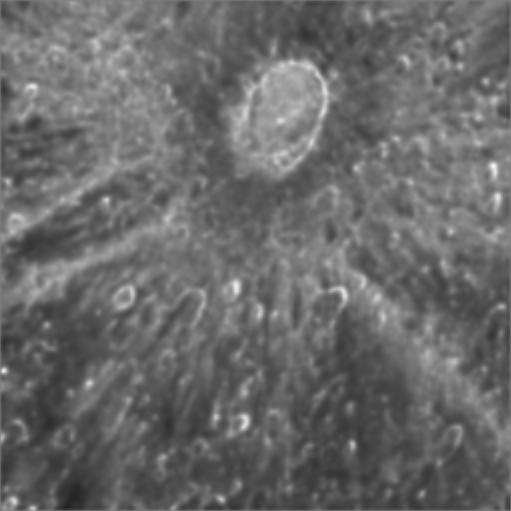} \\[.5em]

    \raisebox{1\height}{\rotatebox[origin=c]{90}{\tiny \texttt{low blur, high noise}}} &
    \includegraphics[width=0.25\textwidth]{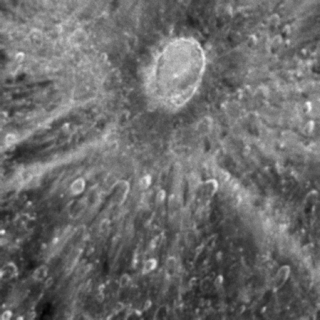} &
    \includegraphics[width=0.25\textwidth]{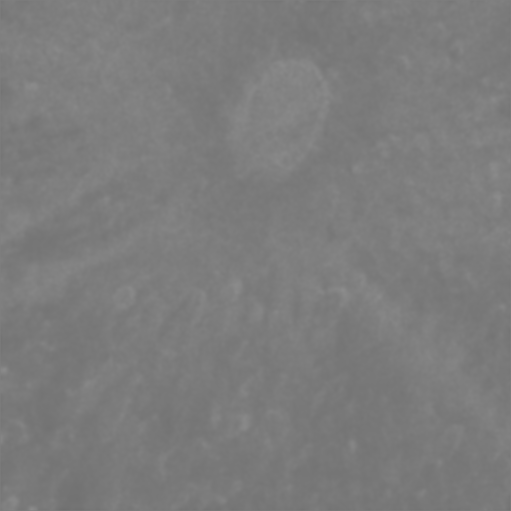} &
    \includegraphics[width=0.25\textwidth]{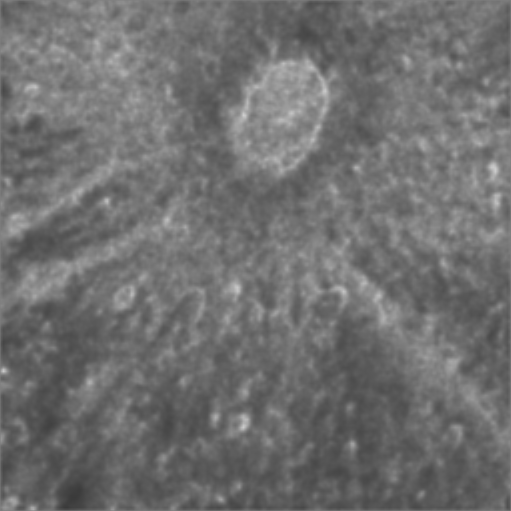} \\[.5em]

    \raisebox{1\height}{\rotatebox[origin=c]{90}{\tiny \texttt{high blur, low noise}}} &
    \includegraphics[width=0.25\textwidth]{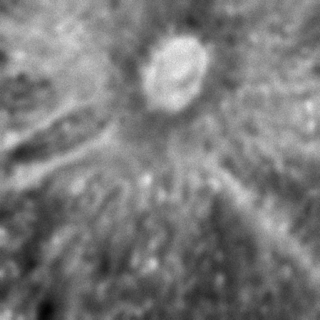} &
    \includegraphics[width=0.25\textwidth]{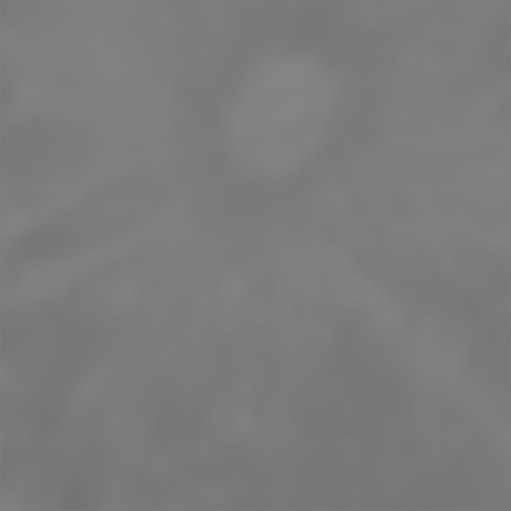} &
    \includegraphics[width=0.25\textwidth]{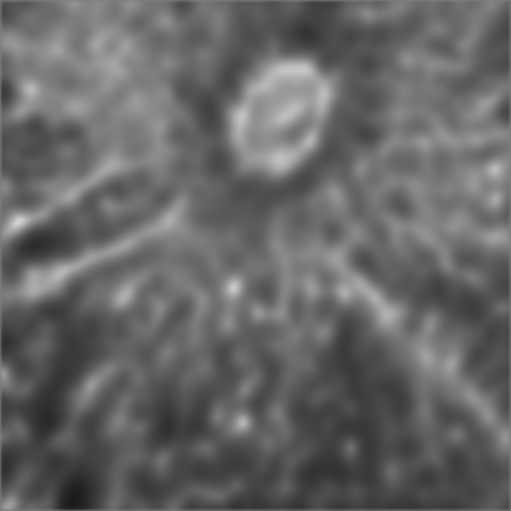} \\[.5em]

    \raisebox{1\height}{\rotatebox[origin=c]{90}{\tiny \texttt{high blur, high noise}}} &
    \includegraphics[width=0.25\textwidth]{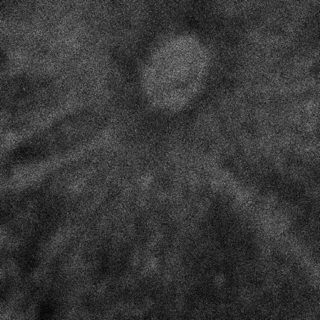} &
    \includegraphics[width=0.25\textwidth]{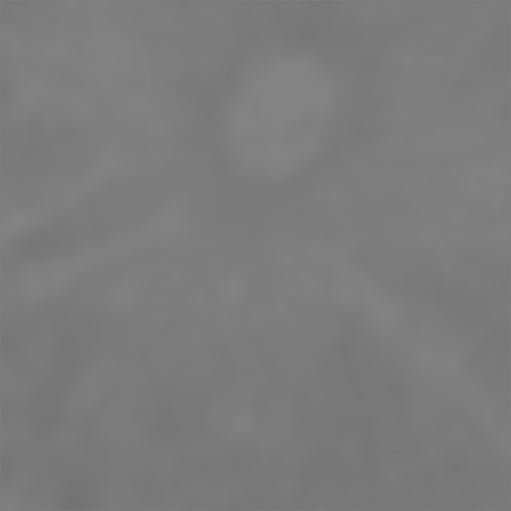} &
    \includegraphics[width=0.25\textwidth]{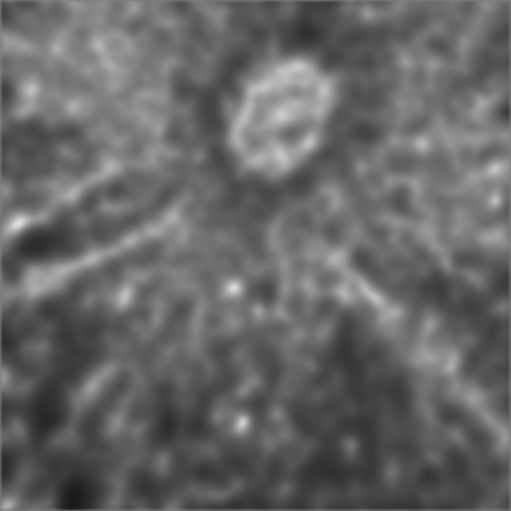}
\end{tabular}%
}
\caption{Deblurring results for the Crater Tycho image under varying noise and blur levels, as specified in \cref{tbl:blur_vs_noise}, after 60 iterations. The ML reconstructions far outperform their single-level counterparts, providing detailed images in few iterations, even in cases of severe degradation.}
\label{fig:images-deconvolution}
\end{figure}

\subsection{Tomographic reconstruction} \label{sec:tomographic_reconstruction}
Another state-of-the-art method for solving inconsistent, nonnegative linear systems is the Simultaneous Multiplicative Algebraic Reconstruction Technique (SMART) \cite{Andersen:1984,Atkinson_Soria_2009} which is particularly popular in tomographic reconstruction due to its efficiency and ability to handle sparse matrices. It can be viewed as a special case of the exponentiated gradient descent method applied to the objective 
\begin{equation}\label{eq:klAxb}
    \KL(Ax,b) = \la Ax, \ln \frac{Ax}{b} \ra - \la \eins, Ax - b \ra.
\end{equation} 
Unlike \eqref{eq:objective-astro-klbAx}, which seeks to match the observed noisy image $b \in \R^m_{++}$ to the expected blurred image $Ax$, \eqref{eq:klAxb} is tailored to ensure that the predicted projections match the actual measurements. SMART provides an efficient reconstruction scheme, particularly when the nonnegative system matrix $A \in \R^{m \times n}_+$ is sparse, as is typical in tomography, often returning meaningful solutions after only a few iterations. To explore its advantages further, we examine the objective with box constraints in the multilevel framework
\begin{equation} \label{eq:objective-tomography-klAxb}
    \min_{x \in [0,1]^n} \KL(Ax,b).
\end{equation}
Since $\KL(x,y) = D_\vphi(x,y)$ for the negative entropy $\vphi(x) = \la x, \ln x \ra - \la \eins, x \ra$, it is natural to choose the negative entropy as a prox function. In fact, \eqref{eq:objective-tomography-klAxb} is $\norm{A}_1$-smooth relative to $\vphi$, cf. \cref{lm:klAxb_rel_smth_to_neg_entropy}, and thus by \cref{prop:nnentropy-to-FD} it is $\norm{A}_1$-smooth relative to the Fermi-Dirac entropy
\begin{equation}
    \vphi_\Box(x) = \sum_{i=1}^n x_i\ln x_i + (1-x_i)\ln(1-x_i).
\end{equation}

\paragraph{Experimental setup} We reconstruct the Walnut phantom (\cref{fig:reference-images}, center) at a resolution of $n = 1023 \times 1023$, subsampled using a tomographic projection matrix $A\in\mathbb{R}^{m\times n}$ with the ASTRA toolbox\footnote{\url{https://astra-toolbox.com/}}. $200$ parallel beam projections are taken at equidistant angles in the range $[0, \pi]$ using $1023$ detectors, yielding $m = 204600$ total projections at an undersampling rate of $20\%$. We initialize with $x^0 = 0.5 \cdot \eins_n$.

\paragraph{Multilevel structure} We use \cref{alg:ML-BPGD} with a total of $3$ levels (coarse grid sizes $511 \times 511$ and $255 \times 255$): one iteration performed on the finest level, $5$ on the middle and $10$ iterations for the coarsest one.
For the coarse levels, we use as many detectors as the width of the coarse image with $100$ equidistant angles in the range of $[0, \pi]$ using parallel beam geometry for an undersampling rate of  $20\%$ and $40\%$ for the two coarse levels, the latter being the coarsest. We use the same transfer operators and coarse correction condition as in \cref{sec:deconvolution} and apply an Armijo line search for coarse corrections across all levels. The performance of the ML-BPGD method in comparison to its single-level counterpart is presented in \cref{fig:tomography}.

\paragraph{Coarse problem construction}
The coarse models are given by
\begin{equation} \label{eq:coarse-model-KLbAx-reconstruction}
    \min_{x \in C_\ell^k} \psi_{\ell}^k(x), \quad \psi_{\ell}^k \text{ as in } \eqref{eq:coarse-model-ML}, \quad f_\ell(x) = \KL(A_\ell x, b_\ell)  .  
\end{equation} 
The constraint sets $C_\ell^k = [l_\ell^k, u_\ell^k]$ are defined by recursive adaptation of the lower and upper bounds, see \cref{appdx:constraints_separable_linear} with initializations $l_0^k = 0$ and $u_0^k = 1$ for all $k \in \N$.
The coarse objectives are then $\norm{A_\ell}_1$-smooth relative to the adapted Fermi-Dirac entropies as defined in \eqref{eq:generalized-FD} with the computed bounds inserted. We use the inverse of the relative smoothness constant as the step size for BPGD iterates. The computation of the B(ounded)-SMART updates for \eqref{eq:objective-tomography-klAxb} and \eqref{eq:coarse-model-KLbAx-reconstruction} is detailed in \Cref{sec:solving_BPGD_iterates_log_barrier}. 

\begin{figure}[ht]
\centering

\begin{minipage}[c]{0.4\columnwidth}
    \centering
    \includegraphics[width=1.\linewidth]{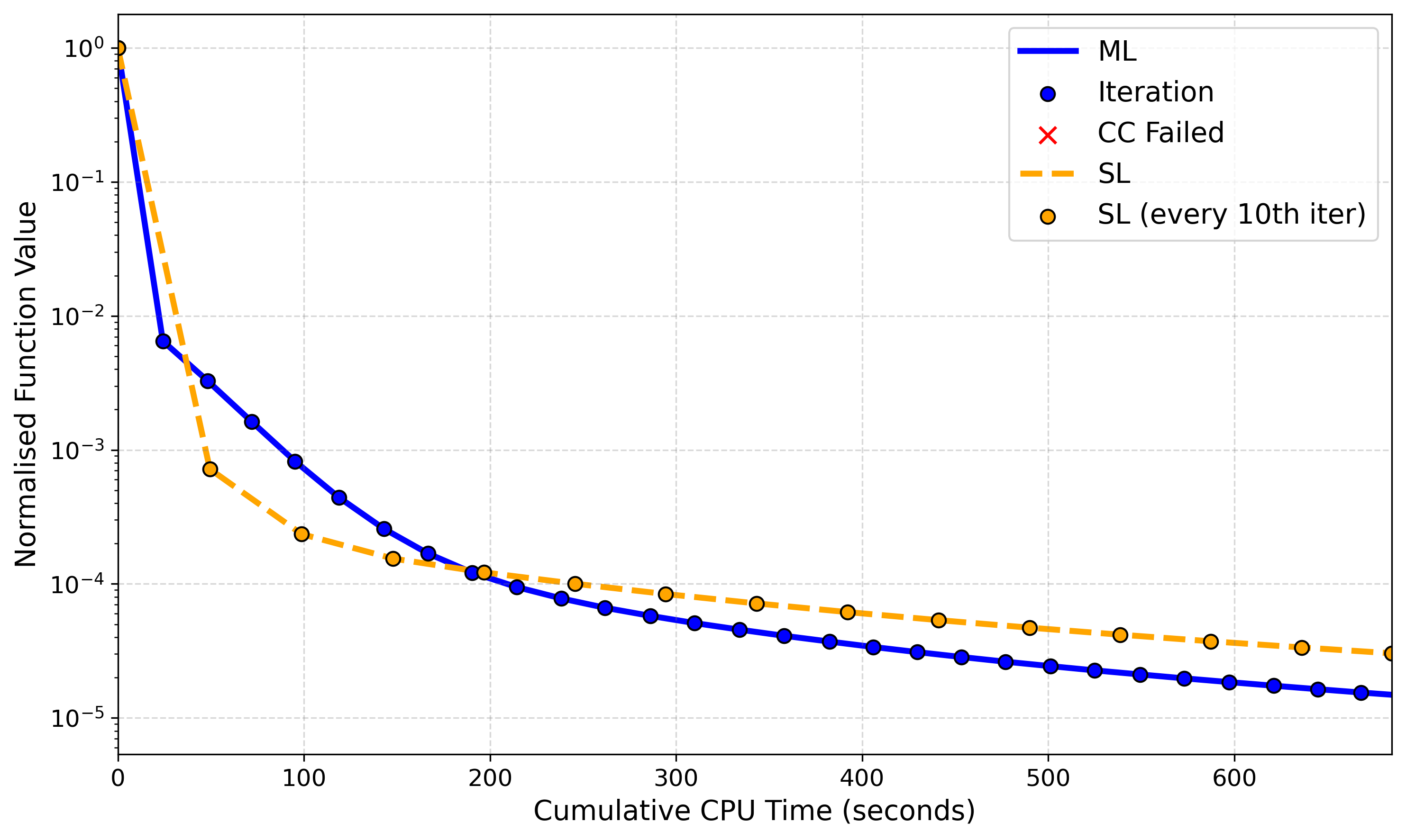}
\end{minipage}
\begin{minipage}[c]{0.5\columnwidth}
    \centering
    \setlength{\tabcolsep}{2pt}
    \begin{tabular}{c c c c c} 
        \texttt{iter1} & \texttt{iter10} & \texttt{iter20} & \texttt{iter30} & \\

        \includegraphics[width=0.2\columnwidth]{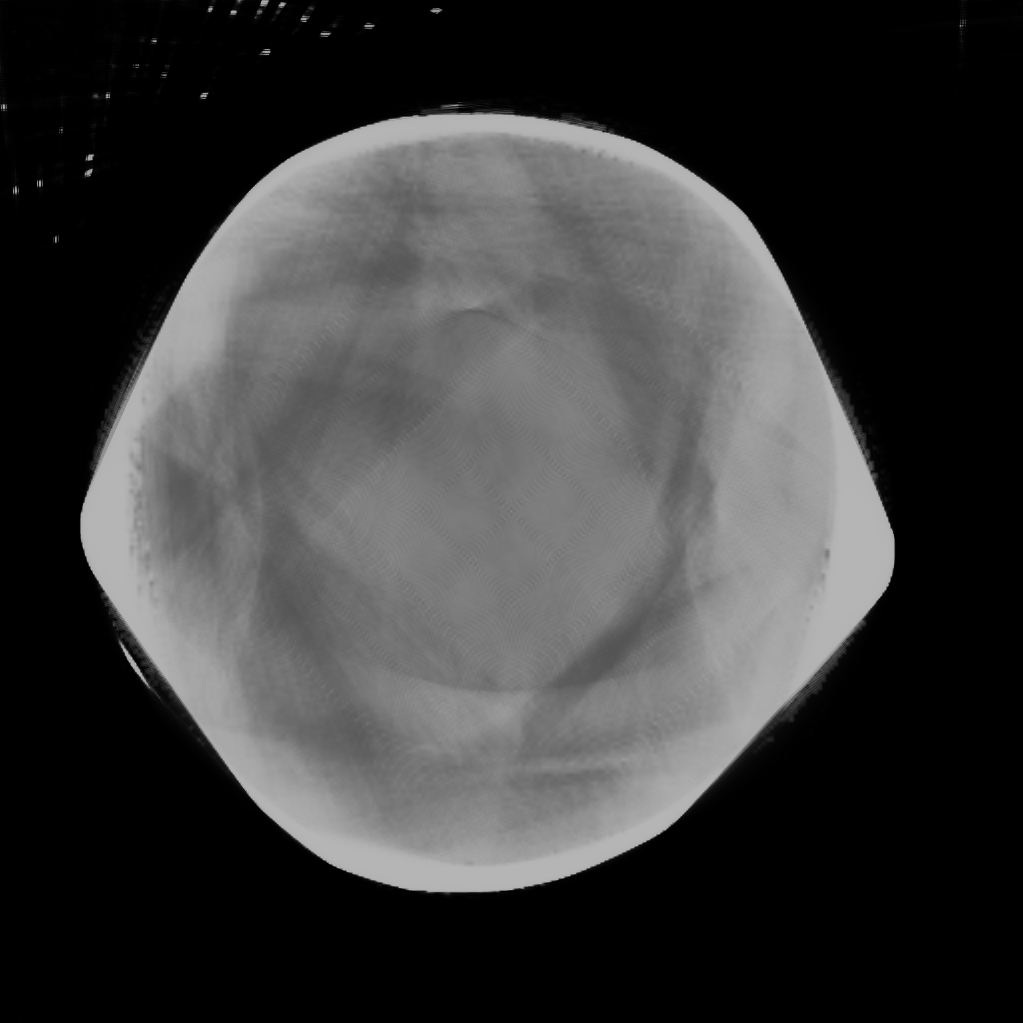} &
        \includegraphics[width=0.2\columnwidth]{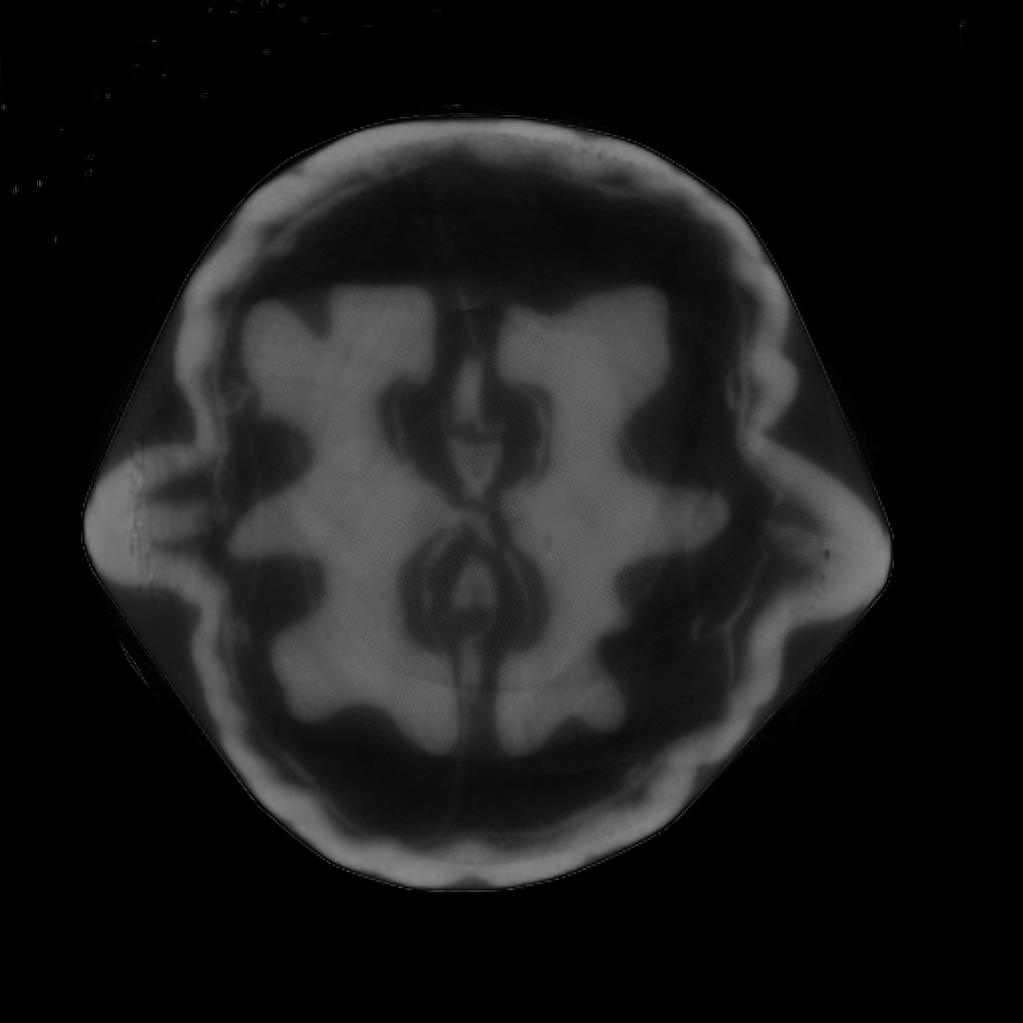} &
        \includegraphics[width=0.2\columnwidth]{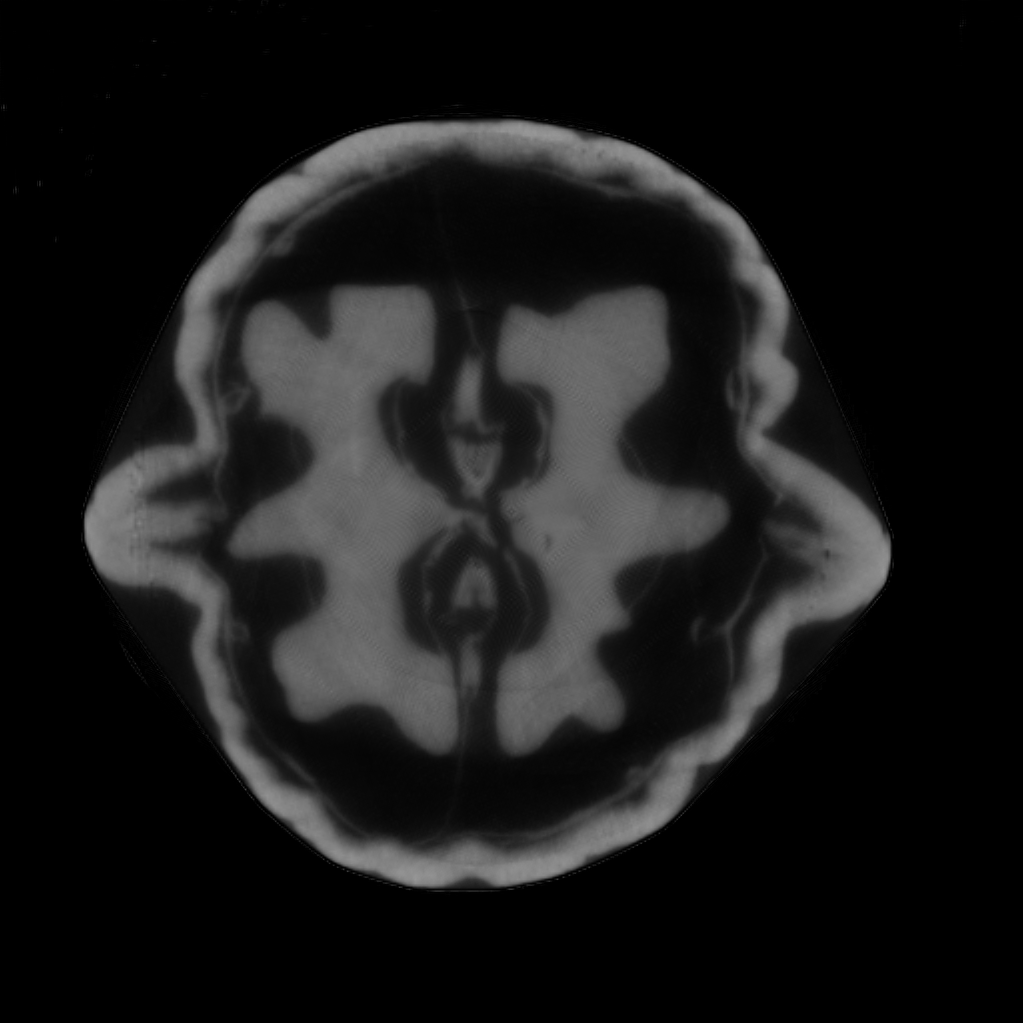} &
        \includegraphics[width=0.2\columnwidth]{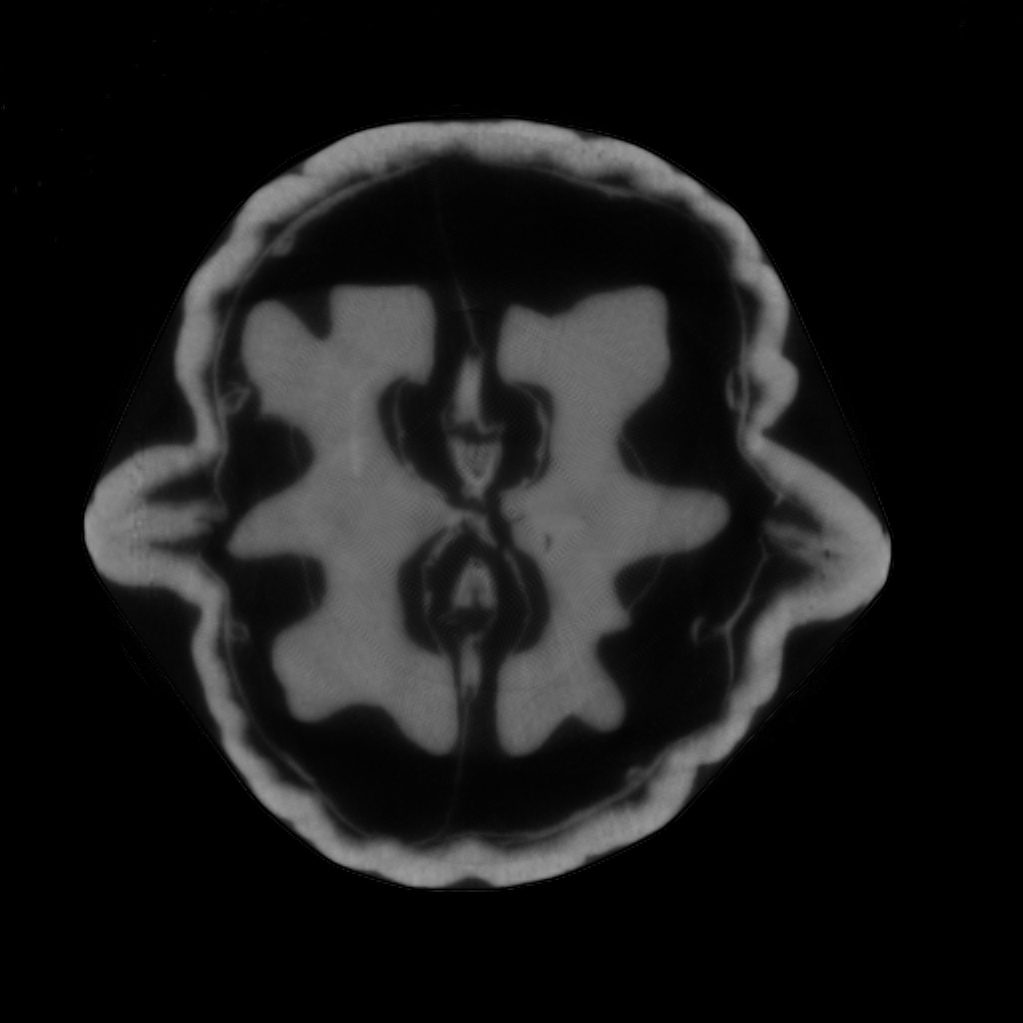} &
        \rotatebox{90}{\quad \; \texttt{ML}}
        \\[-1mm]
        
        \includegraphics[width=0.2\columnwidth]{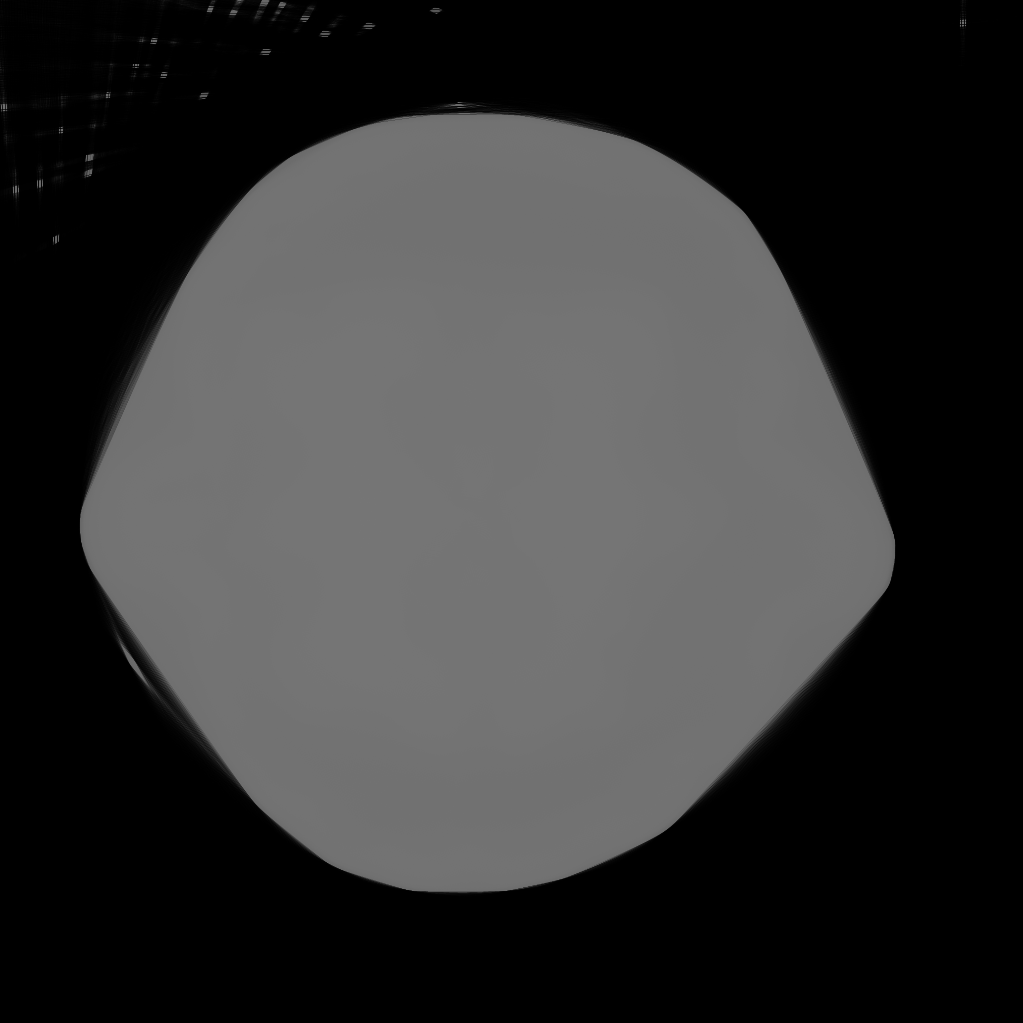} &
        \includegraphics[width=0.2\columnwidth]{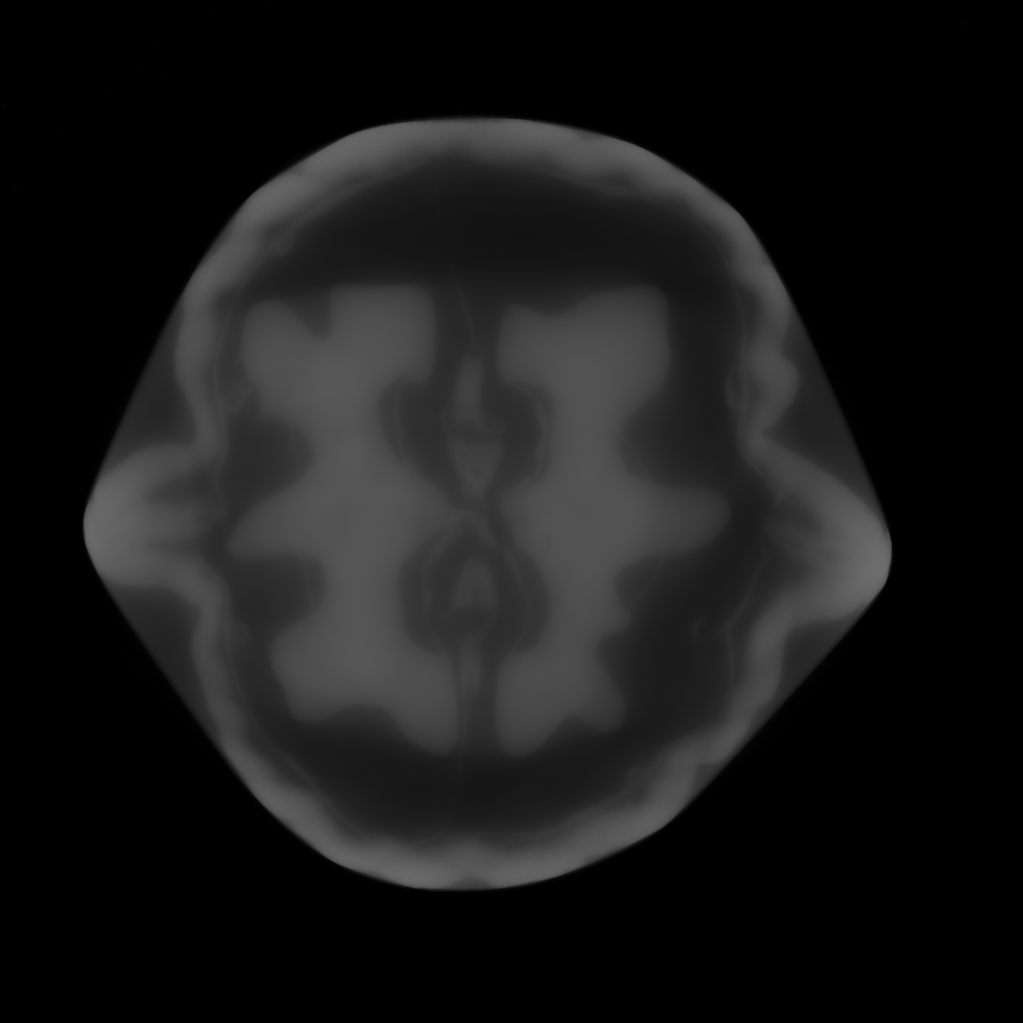} &
        \includegraphics[width=0.2\columnwidth]{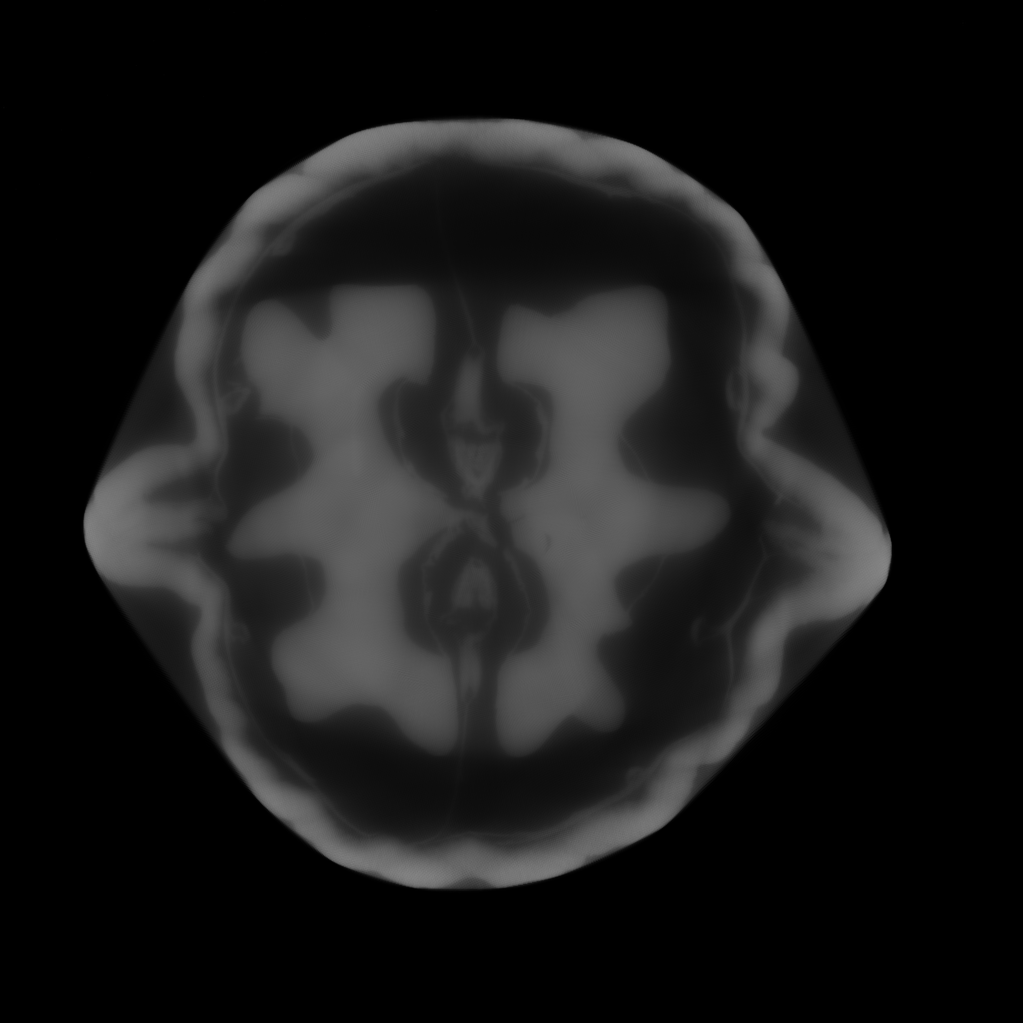} &
        \includegraphics[width=0.2\columnwidth]{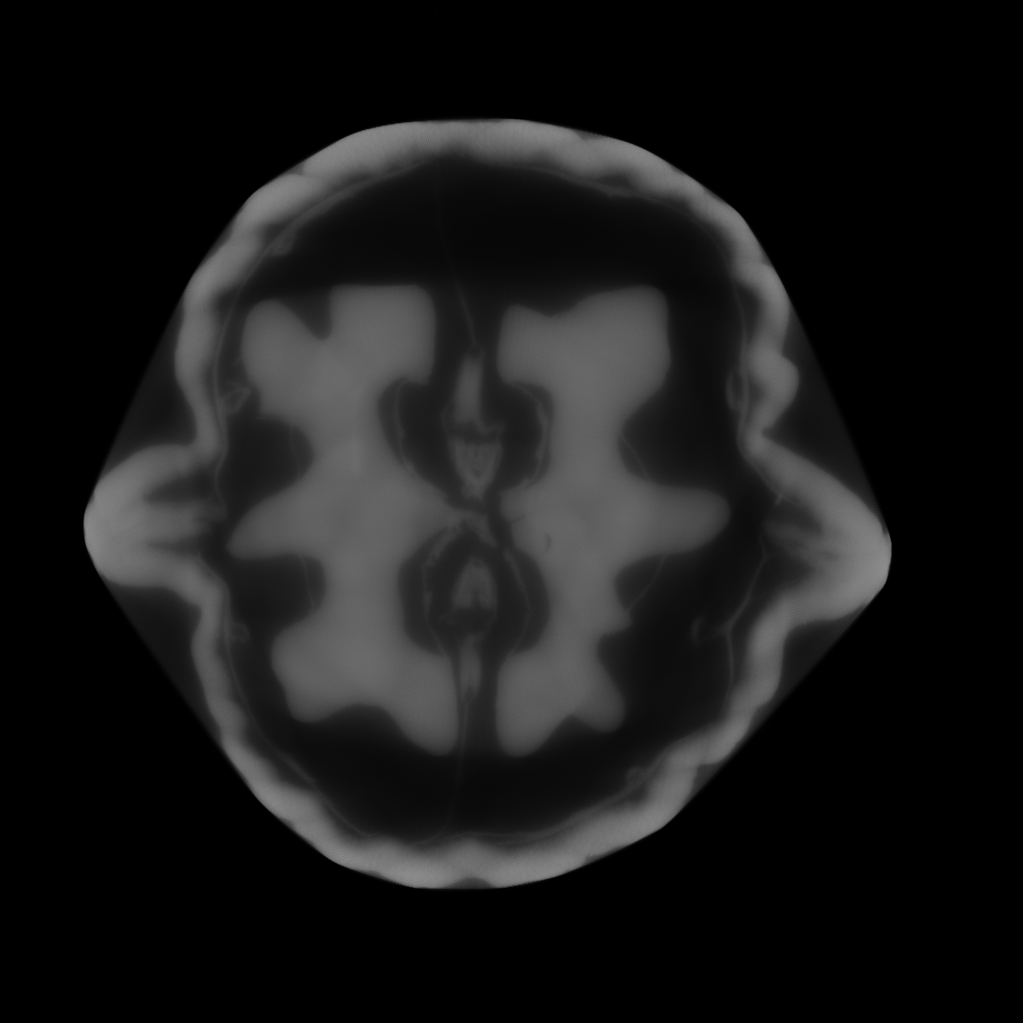} &
        \rotatebox{90}{\qquad \texttt{SL}} \\ 

        \texttt{iter1} & \texttt{iter50} & \texttt{iter100} & \texttt{iter150} &\\ 
    \end{tabular}
\end{minipage}
\caption{
\textbf{Comparison of multilevel vs. single-level BPGD for tomographic reconstruction.} Results are shown for the $1023 \times 1023$ Walnut Phantom from $20\%$ undersampled data. We use three discretization levels.
\texttt{Left:} Normalized objective function values vs. cumulative CPU time (in seconds). The multilevel BPGD (ML-BPGD, blue, with markers at each multilevel iteration) rapidly catches up to the already fast single-level method (SL, yellow, with markers every 10 iterations), achieving comparable objective values in just seven iterations. While each ML iteration incurs extra overhead due to coarse model computations, it shows a more substantial reduction in function value per iteration: approximately four ML iterations match the effect of 10 single-level ones. The coarse correction condition consistently holds up to the final plotted iteration.
\texttt{Right:} Selected iterations show consistently superior reconstructions from ML-BPGD.  Despite higher per-iteration costs, ML achieves comparable visual quality about five times faster than the SL approach. This demonstrates ML-BPGD’s clear advantage in highly undersampled settings, where objective values alone may not fully reflect reconstruction quality.
}
\label{fig:tomography}
\end{figure}

\subsection{D-optimal design} \label{sec:d-optimal-design}
Given a design system matrix $H \in \R^{m \times n}$ of rank $m$, with $n > m$, the D-optimal design problem optimizes the design variables on the $n$-simplex $x \in \Delta^n := \{x \in \R^n: \la \eins, x \ra = 1, x_i \geq 0 \; \forall i \in [n]\}$ of the experimental setup to maximize the information gained about the $m$-dimensional model parameters, \cite{Atwood69:optimaldesign}. It is stated as
\begin{align} \label{eq:fine-model-d-optimal-design}
    \min_{x \in \Delta^n} f(x) = - \ln \det(HXH^\top)
\end{align}
where $X := \Diag(x)$. The D-design problem is $1$-relatively smooth to the log-barrier function $\vphi(x) = - \la 1, \ln x \ra$, compare \cref{lm:d-opt-design_rel_smth_to_log_barrier}.

We adapt this D-optimal design setup to the problem of tomographic reconstruction. For a fixed amount of detectors $d$, the reconstruction matrix $A \in \R^{(d \cdot r) \times (n^2)}$ maps from the unknown internal structure (e.g. an $n \times n$ image) to the $r$ many measured projections at each detector. While the number of available detectors in, e.g., a CT scanner is fixed, the angles could vary. We set up a D-optimal design problem with $H = A^\top$ to identify the importance of each projection angle by maximizing the Fisher information matrix $HXH^\top$, and from this, extract the most informative angles under a sparsity constraint. This, in turn, yields lower reconstruction error and better noise robustness. For a detailed study of D-optimal design in the context of BPGD, see \cite{Lu2018}. Further methods for selecting the optimal experimental design for tomographic reconstruction are studied in \cite{Haldar2018,Fathi2025}

\paragraph{Experimental setup} We use $31$ detectors to measure the importance of $120$ equidistant angles in the range $[0, \pi]$ to reconstruct a $31 \times 31$ image. To model the capabilities of the optimized design under sparsity, we extract the best $15$ angles from the optimized design variables and compare it to the performance of $15$ equidistant angles in the range of $[0, \pi]$ for reconstructing the pixelated Mario image (\cref{fig:reference-images}, right). Here, we use a least-squares objective since optimizing the reconstruction objective is not the target of this experiment. We initialize using uniform weights.

\paragraph{Multilevel structure} We use \cref{alg:ML-BPGD} with only one coarse level, by restricting the number of detectors to match the width of the restricted $15 \times 15$ image.
We do so by employing a 1D linear interpolator defined by the $K_{1\text{D}}$ kernel, see \cref{sec:2l-BPGD-transfer-ops}, applied to the rows of the design variable reshaped into a matrix. We use three iterations on the coarse level and initialize with a uniform weighting. Since the algorithm converges in a few iterations, we simplify the coarse correction criteria to only check the proximity of the current iterate to the last one which triggered a coarse correction using $\norm{x^k - \tilde{x}} \geq \epsilon_x$ with $\epsilon_x = 1e-2$. We again use Armijo line search for the coarse corrections. Comparison of the performance of ML-BPGD to its single-level variant is presented in \cref{fig:d-design}.

\begin{figure}[ht]
\centering

\begin{minipage}[c]{0.5\textwidth}
    \centering
    \includegraphics[width=1.\linewidth]{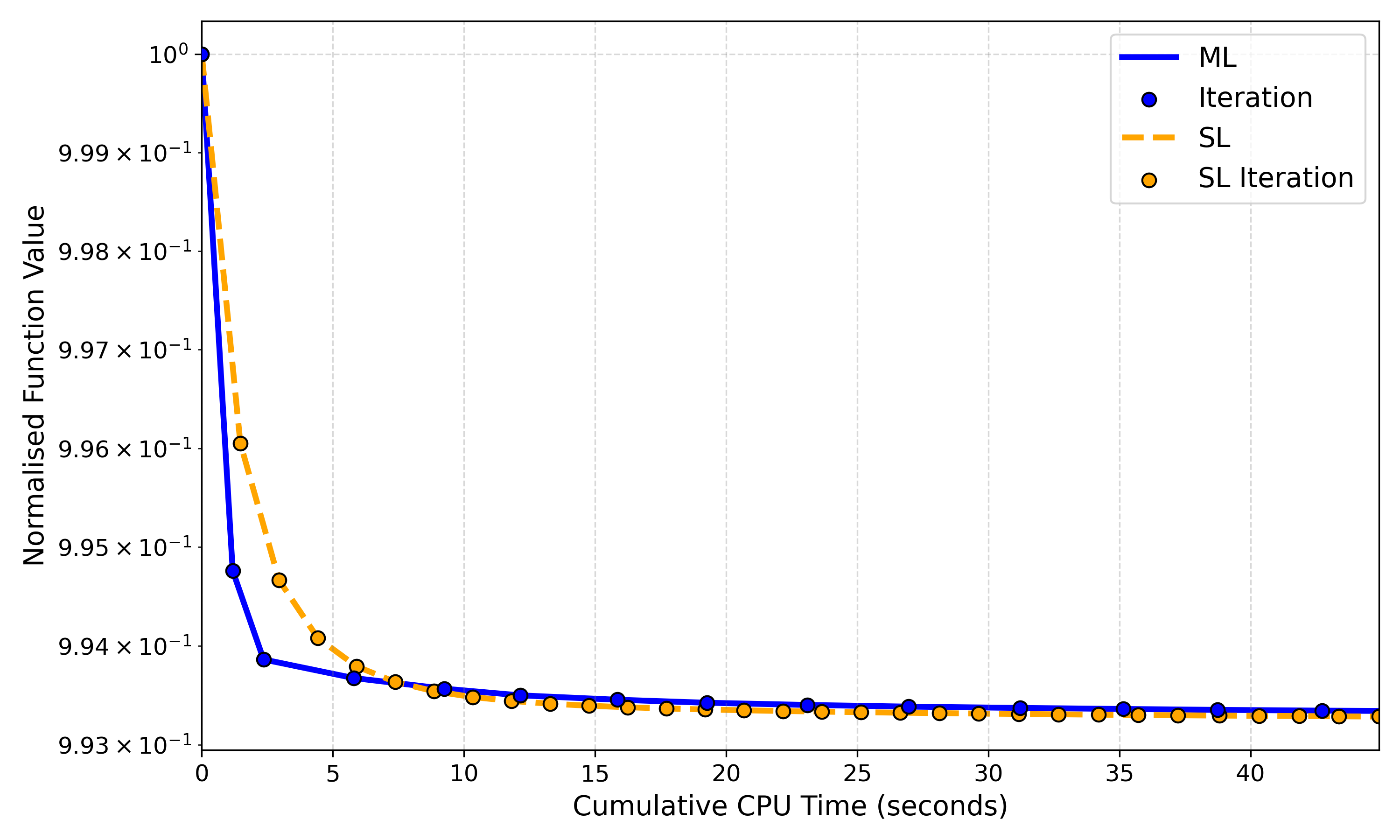}
\end{minipage}
\begin{minipage}[c]{0.35\textwidth}
    \centering
    \includegraphics[width=0.45\linewidth]{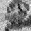}
    \hspace{1mm}
    \includegraphics[width=0.45\linewidth]{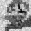}
    \\
    \includegraphics[width=0.45\linewidth]{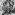}
    \hspace{1mm}
    \includegraphics[width=0.45\linewidth]{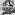}

\end{minipage}

\caption{\textbf{Results for the D-optimal problem for the experimental setup for tomographic reconstruction.} \texttt{Left}: Comparison of Multilevel vs. Single-Level BPGD. The ML-BPGD variant outperforms its single-level counterpart initially. \texttt{Right}: Reconstruction of pixelated mario using $15$ equidistant angles (upper left) and the $15$ Fisher-information maximizing angles (upper right). The bottom row depicts the images downsampled up to 8 channels to match the original image.}
\label{fig:d-design}
\end{figure}

\paragraph{Coarse problem construction} We set up the coarse experimental design matrix $H_1 = A_1^\top \in \R^{(15 ^2) \times (15 \cdot 120)}$. The coarse model is then given by
\begin{equation} \label{eq:coarse-model-d-optimal-design}
    \min_{x \in \Delta^{n_1}(l_1^k, S_1^k)} \psi_{1}^k(x), \quad \psi_{1}^k \text{ as in } \eqref{eq:coarse-model-ML}, \quad f_1(x) = - \ln \det (H_1 X H_1^\top)
\end{equation}
with $\Delta^{n}(l, S) : = \{ x \in \R^{n}: \la \eins, x \ra = S, x_i \geq l \}$ defining the $l$-translated and $S$-scaled probability simplex. For the $k$-th iterate, the lower bound is again recursively computed by the $l_\infty$ argument \eqref{eq: recursive lower bound}, albeit using a different $P$ than the previous numerical experiments. Choosing $S_1^k = \la \eins, x_1^{0,k} \ra$ preserves the consistency of criticality, see \cref{prop:criticality-consistency-simplex}. The solvability of the BPGD iterate for the objectives \eqref{eq:fine-model-d-optimal-design} and \eqref{eq:coarse-model-d-optimal-design} is expanded upon in \cref{subsec:solvability_simplex}, where we use the relative smoothness constant as a step size for all BPGD iterates, with $\tau_\ell \equiv 1$.

\section{Conclusion} \label{sec:conclusions}
We have proposed ML-BPGD, a multilevel extension of Bregman Proximal Gradient Descent for constrained convex optimization problems under relative smoothness. Our approach incorporates coarse-level information to accelerate computation, while explicitly handling constraints at all levels of discretization. We established the well-definedness of the algorithm and provided a convergence guarantee for the function values. The effectiveness of ML-BPGD was demonstrated on large-scale imaging problems with inherent Bregman geometry. Numerical experiments confirm that ML-BPGD achieves a significant acceleration over its single-level counterpart, particularly in the early stages of optimization.

\addtocontents{toc}{\protect\setcounter{tocdepth}{0}}

\section*{Acknowledgments}This work is funded by Deutsche Forschungsgemeinschaft (DFG) under Germany’s Excellence Strategy EXC-2181/1 - 390900948 (the Heidelberg STRUCTURES Excellence Cluster).

\appendix
\addtocontents{toc}{\protect\setcounter{tocdepth}{1}}
\section{Examples of prox functions and relative smoothness} \label{apdx:A}
\subsection{Relative smoothness in Euclidean geometry} For $C = \R^n$, we distinguish two cases. If the objective $f$ has a Lipschitz continuous gradient with an easily computable and numerically efficient constant $L$, then one can choose the quadratic power function $\vphi(x) = \frac{1}{2} \norm{x}_2^2$ as a prox function, which generates the quadratic Euclidean distance as $D_{\frac{1}{2} \norm{\cdot}_2^2}$. This is easily extended to a constrained setting, where the BPGD update reduces to the proximal-gradient version.
However, even in the unconstrained setting, there are many objectives of interest that do not adhere to a Lipschitz gradient. The case where the Hessian grows up 
polynomially in $\norm{x}_2$ is extensively discussed in \cite{Lu2018}, where they construct a suitable polynomial $\vphi$ as a prox function.

\begin{proposition}{\cite[Proposition 2.1]{Lu2018}} \label{prop:Lu-rel-smooth-to-poly}
Suppose $f$ is twice differentiable and satisfies $\|\nabla^2 f(x)\| \allowbreak \leq p_r(\|x\|_2)$, where $p_r(\alpha)$ is an $r$-degree polynomial of $\alpha$. Let $L$ be such that
$p_r(\alpha) \leq L(1 + \alpha^r) \quad \text{for } \alpha \geq 0.$ Then $f$ is $L$-smooth relative to
$ \vphi(x) = \frac{1}{r+2} \|x\|_2^{r+2} + \frac{1}{2} \|x\|_2^2.$
\end{proposition} 
\cref{prop:Lu-rel-smooth-to-poly} is extendable to constraints which have easy to compute projections as shown in the appendix of the aforementioned paper.

\subsection{Relative smoothness in the constrained setting} Lipschitz smoothness of the gradient fails whenever the Hessian of $f$ blows up as one approaches the boundary of the feasible set $C$. This is the case for many objective functions of logarithmic and entropic nature, which cover our examples in \Cref{sec:NumExperiments}. The prox functions are then chosen to best match the geometry of such entropic behavior.
We list the prox functions most often encountered in the literature, see \cite{Bauschke:2017aa,Lu2018,Bauschke:19,Hypentropy2020,AubinFrankowski2022}. Since all of the following functions are separable, we formulate them in the one-dimensional case. The corresponding prox function $\Tilde{\vphi}$ for $n$ dimensions is given by $\Tilde{\vphi}(x) = \sum_{i = 1}^n \vphi(x_i)$.
\begin{itemize}
    \item \textbf{log-barrier function (Burg's entropy)} \; $\vphi(x) = -\ln(x)$, \; $\dom \vphi = \R_{++}$;
    \item \textbf{negative entropy} \; $\vphi(x) = x \ln x - x, \; \dom \vphi = \R_+$ with $0 \ln 0 = 0$;
    \item \textbf{Fermi-Dirac entropy} \; $\vphi(x) = x \ln x + (1-x) \ln(1-x), \; \dom \vphi = [0,1]$;
    \item \textbf{$\beta$-hyperbolic entropy (hypentropy)} \; $\vphi_\beta(x) = x \arcsinh \left( \frac{x}{\beta} \right) + \sqrt{x^2 + \beta^2}, \; \dom \vphi = \R$ for any $\beta > 0$.
    Note that $\vphi_\beta$ interpolates between the negative entropy (as $\beta \to \infty$) and the power function (as $\beta \gg x$).
\end{itemize}
In the following, we present a slight adaptation to the log-barrier function in \Cref{sec:adapt-log-barrier}, to match the constraints of our numerical experiments in \Cref{sec:NumExperiments}. \Cref{sec:rel-smth-for-num-exps} presents the relative smoothness statements used in \Cref{sec:NumExperiments}.
\subsubsection{Example: The doubly-bounded log-barrier function} \label{sec:adapt-log-barrier}
We slightly adapt the log-barrier function $\vphi(x) = - \sum_{i = 1}^n \ln x_i$ to incorporate box boundaries. This is of use when we want the domain of the prox function to match the feasible set exactly. Let $C_\Box := [l,u] \subseteq \R^n$ with $n$-dimensional vectors $l < u$. Define the doubly-bounded log barrier function as $\vphi_\Box(x) := \vphi(x-l) + \vphi(u-x)$, or concretely
    \begin{equation} \label{eq:doubly-bounded-log-barrier}
        \vphi_\Box(x) = -\sum_{i=1}^{n} \left[ \ln (x_i-l_i) + \ln(u_i-x_i)\right].
    \end{equation}
    The following proposition shows that the set of functions that are smooth relative to $\vphi$ is a subset of the $\vphi_\Box$-smooth functions. 
    \begin{proposition} \label{prop:log-barrier-orthant-to-box}
        If $f$ is $L$-smooth relative to $\vphi$ on $\R^n_{+}$ for some $L > 0$, then it is also $L$-smooth relative to $\vphi_\Box$ as given by \eqref{eq:doubly-bounded-log-barrier}. 
    \end{proposition}
    \begin{proof}
    Since $\vphi_\Box$ is separable, the linearity of the Bregman divergence allows to show the statement for $n = 1$ without loss of generality. The linearity also implies $D_{\vphi_\Box}(x,y) = D_\vphi(x-l, y-l)+D_\vphi(u-x, u-y)$. Set $g(l):= D_\vphi(x-l, y-l)$. Then, $g^\prime(l) = \frac{y-l}{x-l} + \frac{x-l}{y-l} - 2$ and since $\frac{1}{t} + t - 2 = \frac{(1-t)^2}{t} \geq 0$ for $t > 0$, it follows that $g^\prime(l) \geq 0$ and with it $g(l) \geq g(0)$. The proposition is an immediate consequence of putting this together.
    \begin{equation}
        D_f(x,y) \leq L D_\vphi(x,y) \leq L(D_\vphi(x-l, y-l) + D_\vphi(u-x, u-y)) = LD_{\vphi_\Box}(x,y).
    \end{equation}
    \end{proof}
A similar result is attainable for the relationship between the Fermi-Dirac entropy and the negative entropy. Even more generally, the Fermi-Dirac entropy can be extended to arbitrary box constraints $C_\Box$ by defining
\begin{equation} \label{eq:generalized-FD}
    \vphi_\Box = \sum_{i = 1}^n (x_i - l_i) \ln(x_i - l_i) + (u_i - x_i) \ln(u_i - x_i).
\end{equation}
\begin{proposition}{\cite[Lemma 3]{Petra2013a}} \label{prop:nnentropy-to-FD}
    If $f$ is $L$-smooth relative to the negative entropy on $\R^n_{+}$ for some $L > 0$, then it is also $L$-smooth relative to $\vphi_\Box$ as defined in \eqref{eq:generalized-FD}.
\end{proposition}

\subsubsection{Relative smoothness of the numerical examples} \label{sec:rel-smth-for-num-exps}
For completeness, we provide the relative smoothness statements for the objectives of our numerical experiments. In the following, let $A \in \R^{m \times n}_{+}$ with nonzero rows, and $b \in \R^m_{++}$.
\begin{lemma}{\cite[Lemma 7]{Bauschke:2017aa}} \label{lm:klbAx_rel_smth_to_log_barrier}
    The Poisson log-likelihood $\KL(b,Ax)$ is $\norm{b}_1$-smooth relative to the log-barrier function $\vphi(x) = - \sum_{i = 1}^n \ln(x_i)$ on $\R^n_{++}$.
\end{lemma}


\begin{lemma}{\cite[Proposition 2.2]{Lu2018}} \label{lm:d-opt-design_rel_smth_to_log_barrier}
    The D-optimal design problem is $1$-smooth relative to the log-barrier function $\vphi(x) = - \sum_{i = 1}^n \ln(x_i)$ on $\R^n_{++}$. 
\end{lemma}

\begin{lemma}{\cite[Lemma 8]{Bauschke:2017aa}}\label{lm:klAxb_rel_smth_to_neg_entropy}
    $\KL(Ax,b)$ is $\norm{A}_1$-smooth relative to the negative entropy $\vphi(x) = \sum_{i = 1}^n \left(x_i\ln x_i - x_i\right)$ on $\R^n_{++}$.
\end{lemma}

\section{Efficient solving of BPGD iterates}
Solving the BPGD subproblems \eqref{eq:BPGD-subproblem}
\begin{equation*}
    x^{+}_\tau = \argmin_{u \in C} \{\la c, u \ra + \vphi(u) \}
\end{equation*}
for $c = \tau \df(x) - \dvphi(x)$ is at the core of \cref{ass:solveability,assump:solvability-coarse-2l,assump:solvability-coarse-ml} ensuring that our multilevel method remains well-defined. The following sections illustrate the solvability of these subproblems for the two main proximity functions we work with, as adapted to our needs in \Cref{sec:NumExperiments}.
\subsection{The negative entropy} \label{sec:solving_BPGD_iterates_neg_entropy}
Let $f$ be $L$-smooth relative to the negative entropy $\vphi(x) =  \sum_{i = 1}^n x_i \ln(x_i) - x_i$. Since $\vphi$ is Legendre on $\R^n_{+}$ with $\nabla \vphi^\ast(y) = e^y$, BPGD is equivalent to a mirror descent (MD) step of the exponentiated gradient
\begin{equation} \label{eq:SMART}
    x^+_\tau = xe^{- \tau \df(x)}.
\end{equation}
\cref{prop:nnentropy-to-FD} allows us to incorporate box-constraints into the negative entropy. The prox function \eqref{eq:generalized-FD} is also Legendre on $C_\Box = [l,u] \subseteq \R^n$ and yields the following MD update
\begin{equation} \label{eq:B-SMART}
        x_\tau^+ = \frac{\frac{x-l}{u-x}-\tau \df(x)}{1 + \frac{x-l}{u-x}}.
\end{equation}
For the special case of $f(x) = \KL(Ax,b)$, the updates \eqref{eq:SMART} and \eqref{eq:B-SMART} are denoted as SMART and B-SMART respectively, compare \cite{Andersen:1984,Petra2013a}.
\subsection{The log-barrier function} \label{sec:solving_BPGD_iterates_log_barrier}
Let $f$ be $L$-smooth relative to the log-barrier function $\vphi(x) = - \sum_{i = 1}^n \ln(x_i)$. Since $\vphi$ is Legendre on $\R^n_{++}$ with $\nabla \vphi^\ast(y) = \frac{1}{y}$, the BPGD is equivalent to the MD update
\begin{equation}
    x_\tau^+ = \left(\frac{1}{x} + \tau \df(x)\right)^{-1}.
\end{equation}
\cref{prop:log-barrier-orthant-to-box} 
enables the incorporation of box constraints  $C_\Box$
into the log-barrier framework by defining the barrier function
$\vphi_\Box(x) = - \sum_{i = 1}^n \left[ \ln(x_i - l_i) + \ln(u_i - x_i)\right]$. However, $\vphi_\Box$ is not a Legendre reference function, as its gradient $\nabla \vphi_\Box(x) = \frac{u-l}{(x-l)(u-x)}$ is non-invertible. Consequently,  the BPGD update
\begin{equation}
    \frac{u-l}{(x^+_\tau-l)(u-x^+_\tau)} = \frac{u-l}{(x-l)(u-x)} - \tau \nabla f(x)
\end{equation}
does not admit a mirror descent interpretation.
In the case of unbounded lower or upper bounds, the respective reference functions $\vphi(u-x)$ and $\vphi(x-l)$ are Legendre and we are again in the MD setting with updates given by
\begin{equation}
    x_\tau^+ = l + \frac{1}{\frac{1}{x-l} + \tau \df(x)}, \quad C = \R^n_{>l}
\end{equation}
and
\begin{equation}
    x_\tau^+ = u - \frac{1}{\frac{1}{u-x} + \tau \df(x)}, \quad C = \R^n_{<u}
\end{equation}
respectively.
\subsubsection{Solvability beyond $\dom \vphi = C$} \label{subsec:solvability_simplex}
We briefly illustrate that the solvability of \eqref{eq:BPGD-subproblem} is not limited to the MD special case. Let $f$ be a differentiable function on the relative interior of the $n$-dimensional probability simplex $\Delta^n := \{x \in \R^n: \la \eins, x \ra = 1, x_i \geq 0 \}$ and assume it is $L$-smooth relative to the 
log-barrier function $\vphi$. We need to solve
    \begin{equation} \label{eq:subproblem-d-optimal-design}
        x^{+}_\tau = \argmin_{u \in C} \{\la c, u \ra - \sum_{i = 1}^n \ln u_i \}, \quad \text{ with }c = \tau \df(x) - \frac{1}{x}, \qquad C = \Delta^n.
    \end{equation}
    This subproblem does not have a closed-form solution but remains easily computable. The first-order optimality conditions of \eqref{eq:subproblem-d-optimal-design} imply that the update takes the form $x_\tau^+ = \frac{1}{c + \xi}$ for some scalar $\xi$ which must satisfy $\sum_{i = 1}^n \frac{1}{c_i + \xi} - 1 = 0$, see \cite{Lu2018} for more details. This equation can be efficiently solved via root-finding. 
    This easily extends to the relative interior of the scaled and translated $n$-dimensional probability simplex
    \begin{equation}
    \Delta^{n}(l,S) := \{ x \in \R^{n}: \sum_{i = 1}^{n} x_i = S, x_i \geq l \}
    \end{equation}
    with $l \in \R^n$ and scalar $S > 0$. Taking $C = \Delta^n(l,S)$, solving \eqref{eq:subproblem-d-optimal-design} requires finding the root of
    \begin{equation}
    d(\xi) := \sum_{i = 1}^{n_\ell} \frac{1}{c_i + \xi} - S
    \end{equation}
    on the interval $\mc{U}:=(a,b)$, $a := -\min_i \{c_i\}, b := \max_i \{\frac{1}{l_i} - c_i\}$ to get the update $x_\tau^+ = \frac{1}{c + \xi}$. Note, that $d(\xi)$ is strictly decreasing on $\mc{U}$ with $d(\xi) \to \infty$ as $\xi \to a$ and $d(\xi) \to -S$ for $\xi \to b$ for $b \gg 0$. Thus, the root is unique and is solvable via a suitable root-finding methods, like Newton method or bisection method. 

\addtocontents{toc}{\protect\setcounter{tocdepth}{0}}

\FloatBarrier
\bibliographystyle{plain}
\bibliography{Literature}
\end{document}